\documentclass[12pt]{article}
\usepackage{geometry}                % See geometry.pdf to learn the layout options. There are lots.
\usepackage{graphicx,color}
\usepackage{amssymb,amsmath,amsthm,mathrsfs}
\usepackage[all,cmtip]{xy}
\usepackage{epstopdf, comment}

\numberwithin{equation}{section}

\newtheorem{theorem}{Theorem}[section]

\newtheorem{lemma}[theorem]{Lemma}
\newtheorem{corollary}[theorem]{Corollary}

\theoremstyle{remark}
\newtheorem*{remark}{Remark}

\theoremstyle{definition}

\DeclareMathOperator{\dist}{dist}

\DeclareMathOperator{\supp}{supp }
\DeclareMathOperator{\sing}{Sing}
\DeclareMathOperator{\inn}{Inn}

\DeclareMathOperator{\out}{Out}

\DeclareMathOperator{\aut}{Aut}
\DeclareMathOperator{\hol}{Hol}

\DeclareMathOperator{\zeros}{zeros}

\DeclareMathOperator{\loc}{loc}

\DeclareMathOperator{\BC}{\mathcal{BC}}
\DeclareMathOperator{\escape}{\escape}

\DeclareMathOperator{\cone}{cone}

\DeclareMathOperator{\interior}{Int}

\DeclareMathOperator{\crit}{crit}

\title{Stable convergence of inner functions}
\author{Oleg Ivrii}
\date{September 2, 2018}

\begin{document}

\maketitle

\begin{abstract}
Let $\mathscr J$ be the set of inner functions whose derivative lies in the Nevanlinna class. In this paper, we discuss a natural topology on $\mathscr J$ where $F_n \to F$ if the critical structures of $F_n$ converge to the critical structure of $F$. We show that this occurs precisely when the critical structures of the $F_n$ are uniformly  concentrated on Korenblum stars. The proof uses  Liouville's correspondence between holomorphic self-maps of the unit disk and solutions of the Gauss curvature equation.
Building on the works of Korenblum and Roberts, we show that this topology also governs the behaviour of invariant subspaces of a weighted Bergman space which are generated by a single inner function.
\end{abstract}

\section{Introduction}

Let $\mathbb{D} = \{z \in \mathbb{C} : |z| < 1\}$ be the unit disk and $\mathbb{S}^1 = \{z \in \mathbb{C} : |z| = 1\}$ be the unit circle.
An {\em inner function} is a holomorphic self-map of the unit disk such that for almost every $\theta \in [0, 2\pi)$, the radial limit $\lim_{r \to 1} F(re^{i\theta})$ exists and has absolute value 1. Let $\inn$ denote the space of all inner functions and $\mathscr J \subset \inn$ be the subspace consisting of inner functions which satisfy
\begin{equation}
\label{eq:finite-entropy}
\lim_{r \to 1} \frac{1}{2\pi} \int_{0}^{2\pi} \log^+ |F'(re^{i\theta})| d\theta < \infty,
\end{equation}
that is, with $F'$ in the Nevanlinna class.
The work of Ahern and Clark \cite{ahern-clark} implies that if $F \in \mathscr J$, then $F'$ admits an ``inner-outer'' decomposition $$F' = \inn F' \cdot \out F'.$$
Intuitively, $\inn F' = BS$ describes the ``critical structure'' of the map $F$ -- the Blaschke factor records the locations of the critical points of $F$ in the unit disk, while the singular inner factor describes the ``boundary critical structure.''
In \cite{inner}, the author proved the following theorem, answering a question posed in \cite{dyakonov-inner}:

\begin{theorem}
\label{main-thm}
Let $\mathscr J$ be the set of inner functions whose derivative lies in the Nevanlinna class. The natural map $$F \to \inn(F') \quad : \quad \mathscr J/\aut(\mathbb{D}) \to \inn/ \, \mathbb{S}^1$$
is injective. The image consists of all inner functions of the form $BS_\mu$ where $B$ is a Blaschke product and $S_\mu$ is the singular factor associated to a measure $\mu$ whose support is contained in a countable union of Beurling-Carleson sets.
\end{theorem}

The above theorem says that an inner function $F \in \mathscr J$ is uniquely determined up to a post-composition with a holomorphic automorphism of the disk by its critical structure and describes all possible critical structures of inner functions. We need to quotient $\inn$ by the group of rotations since the inner part is determined up to a unimodular constant.
To help remember this, note that {\em Frostman shifts or post-compositions with elements of $\aut(\mathbb{D})$ do not change the critical set of a function while rotations do not change the zero set.}

By definition, a {\em Beurling-Carleson set} $E \subset \mathbb{S}^1$ is a closed subset of the unit circle of zero Lebesgue measure whose complement is a union of arcs $\bigcup_k I_k$ with $$\|E\|_{\BC} = \sum |I_k| \log \frac{1}{|I_k|} < \infty.$$
We say that $E \in \BC(N)$ if $\|E\|_{\BC} \le N$. We denote the collection of  all Beurling-Carleson sets by $\BC$.

We will also need the notion of a
{\em Korenblum star} which is the union of Stolz angles emanating from a Beurling-Carleson set $E \subset \mathbb{S}^1$\/:
$$
K_E = B(0, 1/\sqrt{2}) \cup \bigl \{z \in \overline{\mathbb{D}} : 1 - |z| \ge \dist(\hat z, E) \bigr \}.
$$
Here, $\hat z = z/|z|$ while $\dist$ denotes the Euclidean distance.
With the above definition, $K_E \subset \overline{\mathbb{D}}$ is a closed set.
We say that the Korenblum star has {\em entropy} or  {\em norm} $\| E \|_{\BC}$.

We endow $\mathscr J$ with the topology of {\em stable convergence} where $F_n \to F$ if the $F_n$ converge uniformly on compact subsets of the disk to $F$ and the Nevanlinna splitting is preserved in the limit: $\inn F'_n \to \inn F'$, $\out F'_n \to \out F'$. Loosely speaking, our main result says that this occurs if and only if the critical structures of $F_n$ are ``uniformly   concentrated'' on Korenblum stars. A precise statement will be given later in the introduction.
As observed in \cite{inner}, in general, some part of the critical structure may disappear in the limit:
\begin{equation}
\inn F' \ge \limsup_{n \to \infty} \, \inn F'_n.
\end{equation}

\medskip

{\em Examples.} (i) If $F_n$ is a finite Blaschke product of degree $n+1$ which has a critical point at $1 - 1/n$ of multiplicity $n$, and is normalized so that $F_n(0) = 0$, $F'_n(0) > 0$, then the $F_n$ converge to the unique inner function $F_{\delta_1}$ with
critical structure $S_{\delta_1} = \exp\bigl( \frac{z+1}{z-1}\bigr)$.
 More generally, if  $F_n$ has $n$ critical points (of multiplicity one) at $c_k = (1-1/n)e^{ik\theta_n}$, $k=1,2, \dots, n$, and
$n \theta_n \log \frac{1}{\theta_n} \to 0$, then the $F_n$ still converge to $F_{\delta_1}$.

(ii) If $n \theta_n \log \frac{1}{\theta_n} \to \infty$ but $n\theta_n \to 0$, then the $F_n$ converge to the identity even though the critical structures $\inn F_n' \to S_{\delta_1}$.

(iii) For any $0 < c < 1$, one can choose $\theta_n$ appropriately so that the $F_n$ converge to
$F_{c \delta_1}$, the unique inner function with critical structure $S_{c \delta_1}$. In this case, $n \theta_n \log \frac{1}{\theta_n}$ must be bounded away from 0 and $\infty$.

We refer to the three possibilities as {\em concentrating}, {\em totally diffuse} and {\em diffuse} respectively.

\subsection{The Korenblum topology}

A simple ``normal families'' argument shows that $\BC(N)$ is compact. We give a brief sketch of the argument, for  details, we refer the reader to \cite[Lemma 7.6]{HKZ}.
Given a sequence of sets $\{E_n\} \subset \BC(N)$, let $I_n^{(1)}$ denote the longest complementary arc in $\mathbb{S}^1 \setminus E_n$ (in case of a tiebreak, we choose $I_n^{(1)}$ to be one of the longest arcs). We pass to a subsequence so that the $I_n^{(1)}$ converge to a limit $I^{(1)}$. Since there is a definite lower bound for the length $\bigl |I_n^{(1)} \bigr |$, these arcs cannot shrink to a point. We then pass to a further subsequence along which the second longest arcs $I_n^{(2)} \to I^{(2)}$ converge. Continuing in this way, and diagonalizing, we obtain a subsequence of $\{E_n\}$ which converges to a set $E \in \BC(N)$. Note that if $E$ is a finite set, this process would terminate in finitely many steps.
The above argument gives the inequality
\begin{equation}
\label{eq:entropy-drops}
\| E \|_{\BC} \le \liminf_{n \to \infty} \| E_n \|_{\BC}.
\end{equation}
We define the {\em Korenblum topology} on  $\BC$ by specifying that $E_n \to E$ if $E_n$ converges to $E$ in the Hausdorff sense and  $\| E \|_{\BC} = \lim_{n \to \infty}
 \| E_n \|_{\BC}$. Inspired by the work of Marcus and Ponce \cite{marcus-ponce}, we call such sequences     {\em concentrating}\/. However, (\ref{eq:entropy-drops}) could be a strict inequality if  a definite amount of entropy gets trapped in smaller and smaller sets.
  Let us state this phenomenon precisely.
For a Beurling-Carleson $E \subset \mathbb{S}^1$, we define its {\em local entropy} with threshold $\eta > 0$ as
$$
\| E \|_{\BC_\eta} = \sum_{|I| < \eta}  |I| \log \frac{1}{|I|},
$$
where we sum over the connected components of $\mathbb{S}^1 \setminus E$ whose length is less than $\eta$.  Then, (\ref{eq:entropy-drops}) is a strict inequality if and only if $\liminf_{n \to \infty} \| E_n \|_{\BC_\eta} > c > 0$ is bounded below by a constant independent of $\eta$.

Let $M_{\BC(N)}(\mathbb{S}^1)$ denote the class of finite positive measures that are supported on a Beurling-Carleson set of norm $\le N$ and  $M_{\BC}(\mathbb{S}^1)$ denote the collection of measures supported on a countable union of Beurling-Carleson sets.
 There are two natural topologies one can put on $M_{\BC}(\mathbb{S}^1)$. First, one can endow $M_{\BC(N)}(\mathbb{S}^1)$  with the weak topology of measures, and then give
$M_{\BC}(\mathbb{S}^1)$ the {\em inductive limit topology}\/.  Roughly speaking, a sequence of positive measures $\mu_n$ converges to $\mu$ if
up to small error, they converge in $M_{\BC(N)}(\mathbb{S}^1)$. More precisely,
 for any $\varepsilon > 0$, we want there to exist an $N > 0$ and a ``dominated'' sequence $\nu_n \to \nu$ such that for all $n$ sufficiently large,
\begin{enumerate}
\item[(i)] $0 \le \nu_n \le \mu_n$,

\item[(ii)] $\nu_n \in M_{\BC(N)}(\mathbb{S}^1)$,

\item[(iii)] $(\mu_n - \nu_n)(\mathbb{S}^1) < \varepsilon$ and $(\mu - \nu)(\mathbb{S}^1) < \varepsilon$.
\end{enumerate}
We refer to this topology as the {\em topology of weak concentration}\/.  From the point of view of this topology, the  ``opposite'' behaviour is manifested by totally diffuse sequences. We say $\{\mu_n\} \subset M_{\BC}(\mathbb{S}^1)$ is {\em totally diffuse} if for any $N > 0$,
$$
\sup_{E \in \BC(N)} \mu_n(E) \to 0, \qquad \text{as }n \to \infty.
$$

However, in this paper, we will use another topology on $M_{\BC}(\mathbb{S}^1)$, which we call the {\em topology of strong concentration} or the {\em Korenblum topology}\/.
In this topology,
$\mu_n \to \mu$  if for any $\varepsilon > 0$, there exists a  sequence $\nu_n \to \nu$ satisfying (i), (ii) and (iii) such that
$\supp \nu_n \to \supp \nu$ in the Korenblum topology of sets.
 The opposite behaviour to the Korenblum topology is manifested by diffuse sequences. We say that a sequence of measures $\mu_n \to \mu$ is {\em diffuse} if for any $\varepsilon >0$, there exists a $\delta > 0$ such that for any
 threshold $\eta > 0$, we have
$$
\|E\|_{\BC_\eta} < \delta \quad \implies \quad \mu_n(E) < \varepsilon, \qquad  n \ge n_0(\delta, \varepsilon, \eta), \quad  E \in \BC.
$$
It is not difficult to see that a sequence is concentrating if and only if it does not dominate any diffuse sequence with non-zero limit.
Furthermore, one can decompose any weakly convergent sequence $\mu_n \to \mu$ into concentrating and diffuse components,
that is, write $\mu_n = \nu_n + \tau_n$ with $\nu_n \to \nu$ and $\tau_n \to \tau$, where $\nu_n$  is concentrating and  $\tau_n$ is diffuse. Even though there are infinitely many  choices for the sequences $\{\nu_n\}$ and $\{\tau_n\}$, the limits $\nu$ and $\tau$ are uniquely determined by $\{\mu_n\}$.
We leave the verification to the reader.

 We say that a finite positive measure $\mu$ on the closed unit disk belongs to the space
 $M_{\BC(N)}(\overline{\mathbb{D}})$ if its support is contained in a Korenblum star of norm $\le N$, while $\mu \in M_{\BC}(\overline{\mathbb{D}})$ if
its restriction $\mu|_{\mathbb{S}^1} \in M_{\BC}(\mathbb{S}^1)$.
We define the {\em Korenblum topology} on $M_{\BC}(\overline{\mathbb{D}})$ by specifying
that a sequence of measures $\mu_n \to \mu$ converges if it does so weakly, and up to small error,  for all sufficiently large $n$, most of the mass of $\mu_n$ is contained in a Korenblum star $K_{E_n}$, for some sequence of Beurling-Carleson sets  $\{E_n\}$ which converge in $\BC$.
Naturally, we say that a sequence of measures $\mu_n \to \mu$ in $M_{\BC}(\overline{\mathbb{D}})$  is {\em diffuse} if for any $\varepsilon >0$, there exists a $\delta > 0$ such that for any threshold $\eta >0$,  we have
$$
\|E\|_{\BC_\eta} < \delta \quad \implies \quad \mu_n(K_E) < \varepsilon, \qquad  n \ge n_0(\delta, \varepsilon, \eta), \quad  E \in \BC.
$$

\subsection{Two embeddings of inner functions}

To an inner function $I$, we associate the measure
\begin{equation}
\label{eq:inner-measure}
\mu(I) \, = \, \sum (1-|a_i|) \delta_{a_i} + \sigma(I) \, \in \, M(\overline{\mathbb{D}}),
\end{equation}
 where the sum ranges over the zeros of $I$ (counted with multiplicity) and $\sigma(I)$ is the singular measure on the unit circle  associated with the singular factor of $I$. This gives an embedding $\inn/\,\mathbb{S}^1 \to M(\overline{\mathbb{D}})$.
 We say that the measure $\mu$  records the {\em zero structure} of $I$ and write $I_\mu := I$. Clearly, the function $I_\mu$ is uniquely determined up to a rotation.

 We can also embed $\mathscr J/\aut(\mathbb{D}) \to M_{\BC}(\overline{\mathbb{D}})$ by taking $F \to \mu(\inn F')$. This embedding records the {\em critical structure} of $F$.
 We use the symbol $F_{\mu}$ to denote an inner function with $\inn F_\mu' = I_\mu$ and $F_\mu(0) = 0$ (again, such a function is unique up to a rotation).

We can now state our main result:

\begin{theorem}
\label{main-thm2}
The embedding $\mathscr J/\aut(\mathbb{D}) \to M_{\BC}(\overline{\mathbb{D}})$ is a homeomorphism onto its image when  $\mathscr J/\aut(\mathbb{D})$ is equipped with the topology of stable convergence and $M_{\BC}(\overline{\mathbb{D}})$ is equipped with the Korenblum topology.
\end{theorem}

\subsection{Connections with the Gauss curvature equation}

We now give an alternative (and slightly more general) perspective of our main theorem in terms of conformal metrics and nonlinear differential equations.
Given a conformal pseudometric $\lambda(z)|dz|$ on the unit disk with an upper semicontinuous density, its {\em Gaussian curvature} is given by
$$
k_\lambda
=
 - \frac{\Delta \log \lambda}{\lambda^2},
 $$
where the Laplacian is taken in the sense of distributions.  It is well known that the Poincar\'e metric $\lambda_{\mathbb{D}}(z) = \frac{1}{1-|z|^2}$ has constant curvature $-4$. For a holomorphic self-map  of the unit disk $F \in \hol(\mathbb{D}, \mathbb{D})$, consider
the pullback
 $$
 \lambda_{F} := F^*\lambda_{\mathbb{D}} = \frac{|F'|}{1-|F|^2}.
 $$
  Since curvature is a conformal invariant, e.g.~see \cite[Theorem 2.5]{conf-metrics}, it follows that
$
k_{\lambda_F} = -4
$
on $\mathbb{D} \setminus \crit(F)$ where  $\crit(F)$ denotes the critical set of $F$. On the critical set, $\lambda_F$ = 0 while its curvature has $\delta$-masses: $k_{\lambda_F} = -4 - 2\pi \sum_{c \in \crit(F)} \lambda_F(c)^{-2} \cdot \delta_c$.

After the change of variables $u_F = \log \lambda_F$, we naturally arrive at the PDE
\begin{equation}
\label{eq:basic-gce}
   \Delta u = 4 e^{2u} + 2\pi \tilde \nu, \qquad \tilde \nu \ge 0,
   \end{equation}
   where
   $\tilde \nu = \sum_{c \in \crit(F)} \delta_{c}
   $ is an integral sum of point masses.
A theorem of Liouville  \cite[Theorem 5.1]{conf-metrics} states that the correspondence $F \to u_F$ is a bijection between
$$
\hol(\mathbb{D}, \mathbb{D})\,/\aut(\mathbb{D}) \quad \Longleftrightarrow \quad \bigl \{ \text{solutions of (\ref{eq:basic-gce}) with }\tilde{\nu} \text{ integral} \bigr \}.
$$
In principle, Liouville's theorem allows one to translate questions about  holomorphic self-maps of the disk to problems in PDE. In practice, however, it is difficult to find questions that are simultaneously interesting in both settings.

It turns out that the question of describing inner functions with derivative in the Nevanlinna class is related to studying the Gauss curvature equation with {\em nearly-maximal} boundary values
\begin{equation}
\label{eq:generalized-GCE}
   \left\{\begin{array}{lr}
        \Delta u = 4 e^{2u} + 2\pi \tilde \nu, &  \text{in } \mathbb{D}, \\
       u_{\mathbb{D}} - u = \mu, &  \text{on } \mathbb{S}^1,
        \end{array}\right.
\end{equation}
where $u_{\mathbb{D}} = \log \lambda_{\mathbb{D}}$ is the pointwise {\em maximal} solution of (\ref{eq:basic-gce}) in the sense that it dominates every solution of  (\ref{eq:basic-gce}) with any $\tilde \nu \ge 0$. In (\ref{eq:generalized-GCE}), we allow $\tilde \nu \in M(\mathbb{D})$ to be any positive measure on the unit disk which satisfies the {\em Blaschke condition}
 \begin{equation}
\int_{\mathbb{D}} (1-|z|) d\tilde \nu(z) < \infty,
\end{equation}
 and $\mu \in M(\mathbb{S}^1)$ to be any finite positive measure on the unit circle.
 The first equality in (\ref{eq:generalized-GCE}) is understood weakly in the sense of distributions: we require $u(z)$ and $e^{2u(z)}$ to be in $L^1_{\loc}(\mathbb{D})$,
and ask that for any test function $\phi \in C_c^\infty(\mathbb{D})$, compactly supported in the disk,
\begin{equation}
\int_{\mathbb{D}} u \Delta \phi \, |dz|^2 = \int_{\mathbb{D}} 4e^{2u} \Delta \phi \, |dz|^2  + 2\pi \int_{\mathbb{D}} \phi  d\tilde \nu,
\end{equation}
while the second equality expresses the fact that the measures $(u_{\mathbb{D}}-u)(d\theta/2\pi)|_{\{|z|=r\}}$ converge weakly to $\mu$ as $r \to 1$.
If $\mu$ and $\tilde \nu$ are as above, set
\begin{equation}
\label{eq:power-combine}
\omega(z) = \mu(z) + \nu(z) := \mu(z) + \tilde \nu(z)(1-|z|) \in M(\overline{\mathbb{D}}).
\end{equation}

\begin{theorem}
\label{solutions-gce}
Given a measure $\omega = \mu + \nu \in M_{\BC}(\overline{\mathbb{D}})$, the equation (\ref{eq:generalized-GCE}) admits a unique solution, which we denote $u_{\mu, \nu}$ or $u_\omega$.
The solution $u_\omega$ is decreasing in $\omega$, that is, $u_{\omega_1} > u_{\omega_2}$ if $\omega_1 < \omega_2$.
However, if $\omega \notin M_{\BC}(\overline{\mathbb{D}})$ then no solution exists.
\end{theorem}

We endow the space of solutions of (\ref{eq:generalized-GCE}) with the {\em stable topology} where $u_{\omega_n} \to u_{\omega}$ if the  $u_{\omega_n}$ converge weakly to $u_\omega$ and the
  $\omega_n$ converge weakly to $\omega$.
In this setting, our main theorem states:

\begin{theorem}
\label{main-thm4}
The stable topology on the space of solutions with nearly maximal boundary values coincides with the Korenblum topology on $M_{\BC}(\overline{\mathbb{D}})$.
 \end{theorem}

 Theorem \ref{main-thm2} is the restriction of Theorem \ref{main-thm4} to integral measures (we say that a  measure $\omega \in M(\overline{\mathbb{D}})$ is {\em integral} if $\tilde \nu$ is an integral sum of $\delta$-masses while $\mu$ can be anything).
 The connection comes from    \cite[Lemma 3.3]{inner}
 which says that if $F_\omega$ is an inner function with critical structure $\omega$, then
$$u_\omega \, = \, \log \lambda_{F_{\omega}} \, = \, \log \frac{|F_\omega'|}{1-|F_\omega|^2}.$$

\begin{lemma}
\label{two-topologies}
A sequence of functions $\{F_n\} \subset \hol(\mathbb{D}, \mathbb{D})\,/\aut(\mathbb{D})$  converges to $F$ uniformly on compact subsets if and only if $u_{F_n} \to u_F$ weakly on the disk.
\end{lemma}

\begin{proof}
The direct implication is easy since the singularities of the $u_{F_n}$ are integrable. For the reverse implication, suppose that $u_{F_n} \to u_F$ and that $G$ is a subsequential limit of the $F_n$. By the direct implication,  $u_F = u_G$. Liouville's theorem tells us that $F = G$ up to post-composition with an automorphism of the disk.
\end{proof}

 \subsection{Invariant subspaces of Bergman space}

 For a fixed $\alpha > -1$ and $1 \le p < \infty$, consider the weighted Bergman space $A_\alpha ^p(\mathbb{D})$ which consists of  holomorphic functions
 on the unit disk satisfying the norm boundedness condition
 \begin{equation}
 \label{eq:beurling-space-def}
 \|f\|_{A_\alpha^p} = \biggl (  \int_{\mathbb{D}} |f(z)|^p \cdot (1-|z|)^\alpha |dz|^2 \biggr)^{1/p} < \infty.
\end{equation}
 For a function $f \in A^p_\alpha$, let $[f]$ denote the (closed) $z$-invariant subspace generated by $f$, that is the closure of the set $\{ p(z)f(z) \}$, where $p(z)$ ranges over polynomials.
 In the work \cite{korenblum}, Korenblum equipped
  subspaces of $A_\alpha ^p(\mathbb{D})$  with the {\em strong topology} where $X_n \to X$ if any $x \in X$ can be obtained as a limit of a converging sequence of $x_n \in X_n$ and vice versa.

 We focus our attention on a small but important subclass of invariant subspaces which are generated by a single inner function (here, we mean a usual Hardy-inner function rather than a Bergman-inner function). Following \cite{HKZ2}, we refer to such subspaces as of {\em $\kappa$-Beurling-type}.
According to a classical theorem of Korenblum \cite{korenblum-cyc}
 and Roberts \cite{roberts}, the equality $[BS_{\mu_1}] = [BS_{\mu_2}]$ holds if and only if $\mu_1 - \mu_2$ does not charge Beurling-Carleson sets.
Comparing with Theorem \ref{main-thm}, we see that the subspaces of $\kappa$-Beurling-type are in bijection with elements of
 $\mathscr J/\aut(\mathbb{D})$.
 We show that this bijection is a homeomorphism:

\begin{theorem}
\label{main-thm3}
For any $\alpha > -1$ and $1 \le p < \infty$, the strong topology on subspaces of $\kappa$-Beurling-type agrees with the Korenblum topology on $M_{\BC}(\overline{\mathbb{D}})$.
\end{theorem}

In the work \cite{kraus}, Kraus proved that the critical sets of Blaschke products coincide with zero sets of functions in $A^2_1$.
It is therefore plausible that inner functions modulo Frostman shifts are in bijection with the collection of $z$-invariant subspaces of $A^2_1$ satisfying the codimension one property.
The work of Shimorin \cite{shimorin} on the approximate spectral synthesis in Bergman spaces is likely to be of use here.

 \section{The Gauss curvature equation}

Consider the Gauss curvature equation
\begin{equation}
\label{eq:fbe}
- \Delta u = - 4 e^{2u} - 2\pi \tilde \nu, \qquad \tilde \nu \ge 0,
\end{equation}
with free boundary (that is, without imposing any restrictions on the behaviour of $u$ near the unit circle). We say that $u$ is a (weak) {\em solution} if for any non-negative function $\phi \in C_c^\infty(\mathbb{D})$,
\begin{equation}
\label{eq:fbe-weak}
- \int_{\mathbb{D}}  u \Delta \phi \, |dz|^2 = - \int_{\mathbb{D}} 4 e^{2u} \phi \, |dz|^2  - 2\pi \int_{\mathbb{D}} \phi d\tilde{\nu}.
\end{equation}
Naturally, we say that $u$ is a (weak) {\em subsolution} if one has $\le$ in (\ref{eq:fbe-weak}) while the word
{\em supersolution} indicates the sign $\ge$.

\begin{theorem}[Perron method]
\label{perron-method}
Suppose $u$ is a function on the unit disk which is a subsolution of the Gauss curvature equation (\ref{eq:fbe}) with
 free boundary, where $\tilde\nu \ge 0$ is a locally finite measure  on the unit disk.
There exists a unique minimal solution  $\Lambda^{\tilde \nu}[u]$ which exceeds $u$.
If $\overline{u}$ is a supersolution with $\overline{u} \ge u$ then $\overline{u} \ge \Lambda^{\tilde \nu}[u]$.
\end{theorem}

\begin{theorem}
\label{l1theorem}
Given a finite measure $\tilde \nu \ge 0$ on the unit disk
 and $h \in L^\infty(\partial \mathbb{D})$,  the Gauss curvature equation
\begin{equation}
\label{eq:generalized-GCE2a}
   \left\{\begin{array}{lr}
        \Delta u  =  4 e^{2u} + 2\pi \tilde \nu, &  \text{in } \mathbb{D}, \\
       u = h, &  \text{on } \mathbb{S}^1,
        \end{array}\right.
\end{equation}
admits a unique solution. If $u_1$ and $u_2$ are two solutions with $h_1 \le h_2$ and $\tilde \nu_1 \ge \tilde \nu_2$ then  $u_1 \le u_2$ on $\mathbb{D}$.
\end{theorem}

The boundary data $h$ in (\ref{eq:generalized-GCE2a}) is interpreted in terms of weak limits of measures: we require that $h\, d\theta$ is the weak limit of
$u\, d\theta|_{\{|z|=r\}}$ as $r \to 1$.
The uniqueness and monotonicity statements of Theorem \ref{l1theorem} can be easily deduced from Kato's inequality
  \cite[Proposition 6.9]{ponce-book} which states that {\em if $u \in L^1_{\loc}$ and $\Delta u \ge f$ in the sense of distributions with $f \in L^1_{\loc}$, then
$\Delta u^+ \ge f \cdot \chi_{u > 0}$}.
 As usual, $u^+ = \max(u,0)$ denotes the positive part of $u$.

\begin{proof}[Proof of Theorem \ref{l1theorem}: uniqueness and monotonicity]
Since $\tilde \nu_1 \ge \tilde \nu_2$, $$\Delta (u_1-u_2) \ge 4e^{2u_1}-4e^{2u_2}$$ in the sense of distributions.
By Kato's inequality,
$$
\Delta (u_1-u_2)^+ \ge \Delta (u_1- u_2) \cdot \chi_{\{u_1 > u_2\}} = (4e^{2 u_1} -4 e^{2 u_2}) \cdot \chi_{\{u_1 > u_2\}} \ge 0
$$
is a subharmonic function. However, the inequality $h_1 \le h_2$ implies that $(u_1-u_2)^+$ has zero boundary values. The maximal principle shows that $(u_1-u_2)^+ \le 0$ or $u_1
 \le u_2$. The same argument also proves uniqueness.
\end{proof}

In order to not interrupt the presentation, we defer the existence statement in Theorem \ref{l1theorem} to Appendix \ref{sec:perron} and instead explain how to derive Theorem \ref{perron-method} from Theorem \ref{l1theorem}.

Suppose $u$ is a subsolution of (\ref{eq:fbe}). For $0 < r < 1$, we use the symbol
$\Lambda_r^{\tilde \nu}[u]$ to denote the  unique solution of (\ref{eq:fbe})
 on $\mathbb{D}_r = \{z : |z| < r\}$ which agrees with $u$ on $\partial \mathbb{D}_r$. (The function $u$ is bounded on $\partial \mathbb{D}_r$ since it is subharmonic on the disk.)
It may alternatively be described as the minimal solution which dominates $u$ on $\mathbb{D}_r$.
 With this definition, $\Lambda_r^{\tilde \nu}[u]$ does not depend on $\tilde \nu|_{\mathbb{D} \setminus \mathbb{D}_r}$.

\begin{proof}[Proof of Theorem \ref{perron-method}]
As $r \to 1$, the $\Lambda^{\tilde \nu}_r[u]$ form an increasing family of solutions (defined on an increasing family of domains) which are bounded
above by $u_\mathbb{D}$, and therefore they converge to a solution, see Lemma \ref{sol-sequences} below.
 From the construction, it is clear that $\Lambda_r^{\tilde \nu}[u] = \lim_{r \to 1} \Lambda^{\tilde \nu}[u]$ is the Perron hull we seek.

Suppose that $\overline{u} \ge u$ is a dominating supersolution. To show that $\overline{u} \ge \Lambda^{\tilde \nu}[u]$, it suffices to show $\overline{u} \ge \Lambda^{\tilde \nu}_r[u]$ on $\mathbb{D}_r$ for any $0 < r < 1$.
Consider the difference $v = \Lambda^{\tilde \nu}_r[u] -\overline{u}$. We want to show that $v^+$ is identically 0.
 Since $\Delta v \,\ge\,  4e^{2\Lambda^{\tilde \nu}_r[u]} - 4e^{2\overline{u}}$, by Kato's inequality, we have
$$
\Delta v^+ \,\ge\,  (4e^{2\Lambda^{\tilde \nu}_r[u]} - 4e^{2\overline{u}})
\cdot \chi_{\{\Lambda^{\tilde \nu}_r[u] -\overline{u}> 0\}} \, \ge \, 0.
$$
Hence, $v^+$ is a subharmonic function on $\mathbb{D}_r$ with zero boundary values. The maximal principle shows that $v^+ \le 0$  in  $\mathbb{D}_r$ and hence must be identically 0.
This completes the proof.
\end{proof}

\begin{lemma}
\label{sol-sequences}
Suppose $\{u_n\}$ is a sequence of solutions of (\ref{eq:fbe}) with measures $\{\tilde \nu_n\}$. If $u_n \to u$ and $\tilde \nu_n \to \tilde \nu$ weakly on the unit disk, then $u$ is a solution of (\ref{eq:fbe}) with measure $\tilde \nu$.
\end{lemma}

\begin{proof}
Since the $u_n  \le u_{\mathbb{D}}$ are locally uniformly bounded above, the exponentials $e^{2u_n}$ are uniformly bounded. It is now a simple matter to examine the definition of a weak solution (\ref{eq:fbe-weak}) and apply the dominated convergence theorem.
 \end{proof}

\begin{corollary}
\label{lambda-sequences}
Suppose $\{u_n\}$ is a sequence of subsolutions of (\ref{eq:fbe}) with measures $\{\tilde \nu_n\}$. If $u_n \to u$ and $\tilde \nu_n \to \tilde \nu$ weakly on the unit disk, then for any $0 < r < 1$,
$$
\liminf_{n \to \infty} \Lambda_r^{\tilde \nu_n}[u_n] \ge \Lambda_r^{\tilde \nu}[u].
$$
The same statement also holds with $\Lambda$ in place of $\Lambda_r$.
\end{corollary}

\subsection{Generalized Blaschke products}

If $\tilde \nu$ is a measure on the unit disk  satisfying the Blaschke condition
\begin{equation}
\label{eq:blaschke-condition-for-measures}
\int_{\mathbb{D}} (1-|a|) d\tilde \nu(a) < \infty,
\end{equation}
then $\nu(a) := (1-|a|)\tilde \nu(a)$ is a finite measure.  It will be convenient to use both symbols $\nu$ and $\tilde \nu$.
We define the {\em generalized Blaschke product} with zero structure $\nu$ by the formula
\begin{equation}
 B_\nu = \exp \biggl ( \int_{\mathbb{D}} \log \frac{z - a}{1-\overline{a} z} \, d\tilde \nu(a) \biggr ),
\end{equation}
cf.~(\ref{eq:inner-measure}).
While $B_\nu$ may not be a single-valued function on the unit disk, its absolute value and hence  zero set are well-defined. Multiplying $B_\nu$ by a
 singular inner function $S_\mu$, we obtain the {\em generalized inner function} $I_\omega = B_\nu S_\mu$ where $\omega = \mu + \nu$.

The following lemma is well known:

\begin{lemma}
\label{boundary-values}
{\em (i)} For $\nu \in M(\mathbb{D})$, the measures $(\log 1/|B_\nu|)(d\theta/2\pi) \bigl |_{\{|z|=r\}}$ tend weakly to  the zero measure as $r \to 1$.

{\em (ii)} If $\mu \in M(\mathbb{S}^1)$ is a singular measure, then $(\log 1/|S_\mu|)(d\theta/2\pi) \bigl |_{\{|z|=r\}} \to \mu$.
\end{lemma}
The above lemma is stated in \cite[Lemma 3.1]{inner} or \cite[Theorem 1.14]{mashreghi} for integral measures, but the proof works for any measure.
We will also need:
\begin{lemma}
\label{weak-log}
Suppose measures $\omega_n \in M(\mathbb{\overline{D}})$ converge weakly to $ \omega$. Then, $$\log\frac{1}{|I_{\omega_n}|} \to \log\frac{1}{|I_{\omega}|}$$
weakly on the unit disk.
\end{lemma}

We leave the proof as an exercise for the reader.

\subsection{Nearly-maximal solutions}

We now prove  Theorem \ref{solutions-gce} which identifies the space of nearly-maximal solutions of the  Gauss curvature equation with $M_{\BC}(\overline{\mathbb{D}})$. The heavy-lifting has been done in \cite{inner} where Theorem
 \ref{solutions-gce} was proved in the case when $\tilde \nu = 0$.  Here, we explain the extension to general measures $\tilde \nu \ge 0$ satisfying the Blaschke condition (\ref{eq:blaschke-condition-for-measures}).

\begin{proof}[Proof of Theorem \ref{solutions-gce}]
Suppose that $u_{\omega}$ is a nearly-maximal solution of the Gauss curvature equation with data $\omega = \mu + \nu \in M_{\BC}(\overline{\mathbb{D}})$. We claim that
\begin{equation}
\label{eq:fund2}
u_{\omega} = \Lambda^{\tilde \nu} \biggl [u_{\mathbb{D}} - \log \frac{1}{|I_\omega|} \biggr ].
\end{equation}
Since (\ref{eq:fund2}) gives an explicit formula for $u_\omega$,  the
 nearly-maximal solution with data $\omega$ is unique. In view of the monotonicity properties of $\Lambda$,  the fundamental identity  (\ref{eq:fund2}) also shows that
 $u_\omega$ is decreasing in $\omega$.

Consider the function
$$
h = u_{\mathbb{D}} - u_\omega - \log \frac{1}{|I_\omega|}.
$$
Since $\Delta h = 4e^{2u_{\mathbb{D}}} - 4e^{2u_\omega} \ge 0$, $h$ is subharmonic. However, by Lemma \ref{boundary-values}, $h$ tends weakly to the zero measure on the unit circle, and therefore it is negative in the unit disk.
By the definition of the Perron hull,
$$ u_{\omega} \, \ge \,
 \Lambda^{\tilde \nu} \biggl [u_{\mathbb{D}} - \log \frac{1}{|I_\omega|} \biggr ]
 \, \ge \,  u_{\mathbb{D}} - \log \frac{1}{|I_\omega|}.$$
Some rearranging gives
$$u_{\mathbb{D}}  - u_{\omega} \, \le \,
u_{\mathbb{D}}  - \Lambda^{\tilde \nu} \biggl [u_{\mathbb{D}} - \log \frac{1}{|I_\omega|} \biggr ]
\, \le \, \log \frac{1}{|I_\omega|}.$$
Taking the weak limit as $r \to 1$ shows that the Perron hull
$
u_* = \Lambda^{\tilde \nu} \bigl [u_{\mathbb{D}} - \log \frac{1}{|I_\omega|} \bigr ]
$
has ``deficiency'' $\mu$ on the unit circle. Since $u_*$ has ``singularity'' $2\pi \tilde \nu$, it  is also a nearly-maximal solution of the Gauss curvature equation with data $\omega$.
To see that $u_* = u_\omega$, we notice that the difference $u_\omega - u_*$ is a non-negative subharmonic function which tends to the zero measure on the unit circle (and hence must be identically 0).  This proves the claim.

Let $u_\mu$ be the nearly-maximal solution of the Gauss curvature equation $\Delta u = 4e^{2u}$ with ``deficiency'' $\mu \in M_{\BC}(\mathbb{S}^1)$. The existence of $u_\mu$ is non-trivial and was proved in \cite{inner} using the connection with complex analysis provided by the Liouville correspondence.
For any Blaschke measure $\tilde \nu \ge 0$  on the unit disk, the Perron method finds the least solution  of $\Delta u = 4e^{2u} + 2\pi\tilde \nu$ satisfying
$
u_\mu \, \ge \, u \, \ge \, u_{\mu} - \log \frac{1}{|B_\nu|}.
$
By Lemma \ref{boundary-values},  $u$ has the correct boundary behaviour in order to solve (\ref{eq:generalized-GCE}), thereby proving the existence of $u_{\mu, \nu}$.

Conversely, suppose that  $\mu \notin M_{\BC}(\mathbb{D})$. It was proved in \cite{inner} that $u_\mu$ does not exist in this case. To show that $u_{\mu, \nu}$ does not exist for any $\nu \in M(\mathbb{D})$, we argue by contradiction: we use
the existence of $u_{\mu, \nu}$ to construct $u_{\mu}$. To this end, we notice that  $\Lambda^0(u_{\mu, \nu} )$ is a solution of the Gauss curvature $\Delta u = e^{2u}$ which
is squeezed between $u_{\mu, \nu} \le \Lambda^0(u_{\mu, \nu} ) \le u_{\mu, \nu} + \log \frac{1}{|B_\nu|}$, and so must be $u_\mu$ by Lemma \ref{boundary-values}.
\end{proof}

In the proof above, we saw the importance of the formula (\ref{eq:fund2}). In the next lemma, we give
 two variations of this identity.

\begin{lemma}
\label{fundamental-lemma4}
Given two measures $\omega_i = \mu_i + \nu_i \in M_{\BC}(\overline{\mathbb{D}})$, $i=1,2$, we have
\begin{align}
\label{eq:fund3}
 u_{\omega_1 + \omega_2} & =  \Lambda^{\tilde \nu_1 + \tilde \nu_2} \biggl [u_{\omega_1} - \log \frac{ 1}{|I_{\omega_2}|} \biggr], \\
 \label{eq:fund3b}
& =  \Lambda^{\tilde \nu_1 + \tilde \nu_2} \Biggl [ \Lambda^{\tilde \nu_1} \biggl [u_{\mathbb{D}} - \log \frac{1}{|I_{\omega_1}|} \biggr] - \log \frac{1}{|I_{\omega_2}|}  \Biggr ].
\end{align}
\end{lemma}

\begin{proof}
The proof of (\ref{eq:fund3}) is very similar to that of (\ref{eq:fund2}).
Since the quantity on the right side (\ref{eq:fund3}) is a solution of the Gauss curvature equation with ``singularity'' $2\pi(\tilde \nu_1 + \tilde \nu_2)$, we simply need to check that it has the correct ``deficiency''
on the unit circle.
To see this, we observe that it is squeezed by quantities with deficiency $\mu_1 + \mu_2$\,:
$$
u_{\omega_1 + \omega_2} \ge
\Lambda^{\tilde \nu_1 + \tilde \nu_2} \biggl [u_{\omega_1} - \log \frac{ 1}{|I_{\omega_2}|} \biggr ] \ge u_{\omega_1} - \log \frac{ 1}{|I_{\omega_2}|}.$$
We leave it to the reader to justify the first inequality by checking
that
 $u_{\omega_1 + \omega_2} \ge u_{\omega_1} - \log \frac{1}{|I_{\omega_2}|}$ using the argument from the proof of Theorem \ref{solutions-gce}.
Equation (\ref{eq:fund3b}) can be obtained by substituting (\ref{eq:fund2}) into  (\ref{eq:fund3}).
\end{proof}

\section{Concentrating sequences}
\label{sec:concentrating}

In this section, we study concentrating sequences of inner functions. We show:

\begin{theorem}
\label{concentrating-thm}
Suppose $\{F_{\omega_n}\}$ is a sequence of inner functions. If the measures $\omega_n \in M_{\BC}(\overline{\mathbb{D}})$ converge to $\omega$ in the Korenblum topology, then the $F_{\omega_n}$ converge uniformly on compact subsets to $F_\omega$.
\end{theorem}

Actually, we give two proofs of the above theorem. The first proof uses hyperbolic geometry to estimate the derivative of a Blaschke product whose critical structure is supported on a Korenblum star.
The second proof uses PDE techniques and applies to arbitrary sequences of nearly-maximal solutions.

For a Beurling-Carleson set $E \subset \mathbb{S}^1$ and parameters $\alpha \ge 1$, $0 < \theta \le 1$, we define the {\em generalized Korenblum star of order $\alpha$} as
\begin{equation}
\label{eq:gen-korstar}
K^{\alpha}_E(\theta) = \bigl \{z \in \overline{\mathbb{D}} : 1 - |z| \ge \theta \cdot \dist(\hat z, E)^\alpha \bigr \}.
\end{equation}
If $\alpha = 1$ and $\theta =1$, the above definition reduces to the one given earlier: $K_E = K^{1}_E(1)$.
By default we take $\theta = 1$, i.e.~we write $K^\alpha_E = K^\alpha_E(1)$.

\begin{lemma}
\label{korlemma}
Suppose $F(z)$ is an inner function with $F(0) = 0$ whose ``critical structure'' $\mu(\inn F')$ is supported on a Korenblum star $K_E$ of order 1 and ``critical mass'' $\mu(\inn F')(\overline{\mathbb{D}}) < M$.
Then,
\begin{equation}
\label{eq:korlemma}
\frac{1-|F(z)|}{1-|z|} \le C(M) \cdot \dist(z, E)^{-4}, \qquad z \in \mathbb{D} \setminus K_E^{4},
\end{equation}
where $\dist$ denotes Euclidean distance.
\end{lemma}

Under the assumptions of the above lemma, we have:

\begin{corollary}
\label{korstar4}
The ``zero structure''
  $\mu(F)$ is supported on a higher-order Korenblum star $K_E^{4}(\theta)$ where $\theta$ is a parameter which depends on $M$. In particular, by Schwarz reflection, $F$ extends to an analytic function
on $\mathbb{C} \setminus {\mathbf r}(K_E^{4}(\theta))$ where ${\mathbf r}(z) = 1/\overline{z}$ denotes the reflection in the unit circle.
\end{corollary}

\begin{corollary}
\label{korstar1}
For a point $\zeta  \in \mathbb{S}^1$ on the unit circle,
$$
|F'(\zeta)| \le C(M) \cdot \dist(\zeta, E)^{-4}.
$$
\end{corollary}

With help of Lemma \ref{korlemma} and its corollaries, the proof of  Theorem \ref{concentrating-thm} runs as follows:

\begin{proof}[Proof of Theorem \ref{concentrating-thm}]

{\em Special case.}
We first prove the theorem in the special case
 when each measure ${\omega_n}$ is supported on a Korenblum star $K_{E_n}$ and the sets
  $E_n \to E$ converge in $\BC$.

   To simplify the notation, let us write $F_n$ instead of $F_{\omega_n}$.
 By a normal families argument and Corollary \ref{korstar4}, we may assume that $F_n \to F$ converge locally uniformly on $\mathbb{C} \setminus {\mathbf r}(K_E^{4}(\theta))$, where the domains of definition $\mathbb{C} \setminus {\mathbf r}(K_{E_n}^{4}(\theta))$ are changing but converge to $\mathbb{C} \setminus {\mathbf r}(K_E^{4}(\theta))$. In this case, the  derivatives $F'_n \to F'$  also converge locally uniformly.

According to \cite[Section 4]{inner}, to show that the sequence $\{F_n\}$ is stable,  it suffices to check that the outer factors converge at the origin:
$$
\lim_{n \to \infty} \int_{\mathbb{S}^1} \log|F'_n| d\theta = \int_{\mathbb{S}^1} \log|F'| d\theta.
$$
The proof will be complete if we can argue that the functions $\log |F'_n|$ are uniformly integrable on the unit circle.
This means that for any $\varepsilon > 0$,
 there exists a $\delta > 0$ so that $\int_A \log |F_n'| d\theta < \varepsilon$ for any $n$, whenever $A \subset \mathbb{S}^1$ is a measurable set with $m(A) < \delta$.
This estimate is provided by Corollary \ref{korstar1} above and the definition of a concentrating sequence
of Beurling-Carleson sets.

\medskip

{\em General case.} According to the definition, $\omega_n \to \omega$ in the Korenblum topology if for any $\varepsilon > 0$, one can find
a concentrating sequence $\omega_n^N \to \omega^N$ in $M_{\BC(N)}(\overline{\mathbb{D}})$
with $0 \le \omega^N_n \le \omega_n$,  $(\omega_n - \omega_n^N)(\overline{\mathbb{D}}) < \varepsilon$ and $(\omega - \omega^N)(\overline{\mathbb{D}}) < \varepsilon$.
According to Lemma  \ref{fundamental-lemma4},
$$
 \log \frac{1}{|I_{\omega_n - \omega_n^N}|} \, \ge \,
 u_{\omega_n^N} - u_{\omega_n}\, \ge \,  0 , \qquad n = 1,2, \dots$$
From the special case of the theorem,  we know that
$
u_{\omega_n^N}
$
converge weakly to $u_{\omega^N}$.
Since the difference $\omega_n - \omega_n^N$ can be made arbitrarily small by making $N$ large,
 Lemma \ref{weak-log} implies that $u_{\omega_n} \to u_\omega$ weakly.
Finally, by Lemma \ref{two-topologies}, this is equivalent to the uniform convergence of  $F_{\omega_n} \to F_\omega$ on compact subsets of the unit disk.
\end{proof}

\subsection{Blaschke products as approximate isometries}
\label{sec:bp-ai}

To prove Lemma \ref{korlemma}, we use the following principle: {\em away from the critical points, an inner function is close to a hyperbolic isometry}. Our discussion is inspired by the work of McMullen \cite[Section 10]{mcmullen-rtree} which deals with finite Blaschke products of fixed degree. Here, we require ``degree independent'' estimates.
To this end, given an inner function $F(z)$,  we consider the quantity
\begin{equation}
\label{eq:mob-inv}
 \gamma_F(z) =  \log \frac{1}{ |\inn F'(z)|}
\end{equation}
which measures how much $F$ deviates from a M\"obius transformation near $z$.
The quantity $\gamma_F$ satisfies the M\"obius invariance relation
\begin{equation}
\gamma_{M_1 \circ F \circ M_2}(z) = \gamma_F(M_2(z)), \qquad M_1, M_2 \in \aut(\mathbb{D}),
\end{equation}
which follows from the identity $\inn \bigl [ (M_1 \circ F \circ M_2)' \bigr ] = \inn F' \circ M_2$.
Let $G(z,w)$ denote the Green's function on the unit disk. When $w=0$, $G(z,0) = \log \frac{1}{|z|}$. If
the singular measure $\sigma(F')$ is trivial (e.g.~if $F$ is a finite Blaschke product), the above definition reduces to
\begin{equation}
\gamma_F(z) = \sum_{c \in \crit(F)} G(z, c).
\end{equation}
For two points $x, y \in \mathbb{D}$, we write $d_{\mathbb{D}}(x,y)$ for the hyperbolic distance and denote the segment of the
hyperbolic geodesic that joins $x$ and $y$ by $[x, y]$. There is a convenient way to estimate  hyperbolic distance.
Let $z \in [x,y]$ be the point closest to the origin. If $z = x$ or $z = y$, then $d_\mathbb{D}(x,y) = d_\mathbb{D}(|x|,|y|)+\mathcal O(1)$ is essentially the ``vertical distance'' from $x$ to $y$. If $z$ lies strictly between $x$ and $y$, then $d_\mathbb{D}(x,y) = d_\mathbb{D}(|x|,|z|)+d_\mathbb{D}(|z|,|y|)+\mathcal O(1)$.

\begin{lemma}[cf.~Proposition 10.9 of \cite{mcmullen-rtree}]
Suppose $F \in \mathscr J$ is an inner function with derivative in the Nevanlinna class. At a point $z \in \mathbb{D}$ which is not a critical point of $F$,
 the 2-jet of $F$ matches the 2-jet of a hyperbolic isometry with an error of
$\mathcal O(\gamma(z))$.
\end{lemma}

\begin{proof}
By  M\"obius invariance (\ref{eq:mob-inv}), it suffices to consider the case when $z = F(z) = 0$ and $F'(0) > 0$. Set
$\delta = \gamma_F(0)$. To prove the lemma, we need to show that
$|F'(0) - 1| = |F''(0)| = \mathcal O(\delta)$.
The definition of $\gamma_F(0)$ implies that $1 - |(\inn F')(0)| \le \delta$.
Applying  \cite[Lemma 2.3]{inner} gives the desired estimate for the first derivative:
$$F'(0) \, = \, \lambda_F(0)  \, \ge \,  |\inn F'(0)| \cdot \lambda_{\mathbb{D}}(0) \, \ge \, 1 - \delta.$$
By the Schwarz lemma applied to $F(z)/z$, we have
 $d_{\mathbb{D}} \bigl (F(z)/z, F'(0) \bigr ) = \mathcal O(1)$ for $z \in B(0,1/2).$
Taking note of the location of $F'(0) \in \mathbb{D}$, this estimate can be written as $|F(z) - z| = \mathcal O(\delta)$ for  $z \in B(0,1/2)$. Cauchy's integral formula now gives the estimate for the second derivative.
\end{proof}

\begin{corollary}[cf.~Theorem 10.11 and Corollary 10.7 of \cite{mcmullen-rtree}]
Suppose $F(z)$ is a finite Blaschke product and $[z_1,z_2]$ is a segment of a hyperbolic geodesic. If for each $z \in [z_1, z_2]$, $\gamma_F(z) < \varepsilon_0$ is sufficiently small, then
$F(z_1) \ne F(z_2).$
In fact, for any $\delta > 0$, we can choose $\varepsilon_0 > 0$ small enough to guarantee that
\begin{equation}
(1 - \delta) \cdot d_{\mathbb{D}}(z_1, z_2) \, \le \,
  d_{\mathbb{D}} \bigl (F(z_1), F(z_2) \bigr ) \, \le \, d_{\mathbb{D}}(z_1, z_2).
\end{equation}
\end{corollary}

\begin{proof}[Sketch of proof]
If we choose $\varepsilon_0 > 0$ small enough, then
$F|_{[z_1,z_2]}$ is so close to an isometry
that the geodesic curvature of its image is nearly $0$. But a path in hyperbolic
space with geodesic curvature less than 1 (the curvature of a
horocycle)
cannot cross itself, so $F(z_1) \ne F(z_2)$. Similar reasoning gives the second statement.
\end{proof}

\begin{remark}
If $\gamma_F(z)$ decays exponentially along $[z_1,z_2]$, i.e.~satisfies a bound of the form $\gamma_F(z) < M \exp \bigl (-d_{\mathbb{D}}(z,z_1) \bigr ),$ for some $M>0$, then McMullen's argument gives the stronger conclusion
$$d_{\mathbb{D}}(F(z_1), F(z_2)) =  d_{\mathbb{D}}(z_1, z_2) + \mathcal O(1).$$
See the proof of \cite[Theorem 10.11]{mcmullen-rtree}.
\end{remark}

 \begin{lemma}
 \label{gammaF-estimate2}
Suppose $I$ is an inner function whose zero structure $\mu(I)$ is contained in a Korenblum star $K_E$ and its zero mass
$\mu(I)(\overline{\mathbb{D}}) < M$. Then, $|I(z)| > c(M) > 0$ is bounded from below on $\mathbb{D} \setminus K_E^2$. More precisely,
$$
\log \frac{1}{|I(z)|} \lesssim  M \exp\bigl(-d_{\mathbb{D}}(z,K_E^2)\bigr), \qquad z \in \mathbb{D} \setminus K_E^2.
$$
\end{lemma}

 For a point $z \in \mathbb{D}$ and an integer $n > 0$ such that $z \notin K_E^n$,  let $z_n$ denote the unique point of intersection of $[0,z]$ and $\partial K_E^n$.

\begin{proof} We may assume that $I$ is a finite Blaschke product as the general case follows by approximating
$I$ by finite Blaschke products whose zero sets are contained in $K_E$.

Let $a$ be a zero of $I$. Elementary hyperbolic geometry and the triangle inequality show that the hyperbolic distance
\begin{align*}
d_{\mathbb{D}}(z, a) & = d_{\mathbb{D}}(z, z_1) + d_{\mathbb{D}}(z_1, a) - \mathcal O(1), \\
& \ge d_{\mathbb{D}}(z, z_1) + d_{\mathbb{D}}(0,a) - d_{\mathbb{D}}(0,z_1) - \mathcal O(1), \\
& \ge d_{\mathbb{D}}(z, z_2) + d_{\mathbb{D}}(0,a) - \mathcal O(1).
\end{align*}
In other words, the Green's function
$$
G(z, a) \lesssim G(0, a) \exp \bigl (-d_{\mathbb{D}}(z,z_2) \bigr )
$$
decays exponentially quickly in the hyperbolic distance $d_{\mathbb{D}}(z,z_2)$.
If $a \in B(0,1/2)$, we instead use the ``trivial'' estimate $G(z,a) \lesssim \exp \bigl ( -d_{\mathbb{D}}(z,0) \bigr ) \le
\exp \bigl(-d_{\mathbb{D}}(z,z_2)\bigr)$. Combining the two inequalities, we get
$$
G(z, a) \lesssim G^*(0, a) \exp \bigl (-d_{\mathbb{D}}(z,z_2) \bigr )
$$
where $G^*(z,w) := \min \bigl (G(z,w),1 \bigr)$ is the truncated Green's function.
Summing over the zeros of $I$ gives
$$
\log \frac{1}{|I(z)|} = \sum_{a \in \zeros(I)} G(z,a) \lesssim M \exp \bigl (-d_{\mathbb{D}}(z,z_2)\bigr) \asymp M \exp \bigl (-d_{\mathbb{D}}(z,K_E^2) \bigr),
$$
where in the second step we made use of  $ \sum_{a \in \zeros(I)} G^*(0,a) \asymp \mu(I)(\overline{\mathbb{D}}) \le M$. This proves the lemma.
\end{proof}

 \begin{corollary}
 \label{gammaF-estimate}
Suppose $F$ is an inner function which satisfies the hypotheses of Lemma \ref{korlemma}.
 For $z \in \mathbb{D} \setminus K_E^2$, the characteristic
$
\gamma_F(z) \lesssim  M\exp\bigl(-d_{\mathbb{D}}(z,K_E^2)\bigr).
$
\end{corollary}

With these preparations, we can now prove Lemma \ref{korlemma}:

\begin{proof}[Proof of Lemma \ref{korlemma}]
Suppose $z \in \mathbb{D} \setminus K_E^{4}$.
 Divide $[0,z]$ into two parts: $[0,z_2]$ and $[z_2,z]$.
By the  Schwarz lemma,
$$
d_\mathbb{D}(F(0),F(z_2)) \le d_{\mathbb{D}}(0,z_2).
$$
However, since $F$ restricted to $[z_2,z]$ is close to a hyperbolic isometry,
\begin{equation}
\label{eq:can-be-improved}
d_{\mathbb{D}}(F(z_2), F(z)) \ge  d_{\mathbb{D}}(z_2, z) - \mathcal O(1).
\end{equation}
The triangle inequality gives
$$
d_{\mathbb{D}}(F(0), F(z)) \ge  d_{\mathbb{D}}(z_4, z) - \mathcal O(1),
$$
which is equivalent to (\ref{eq:korlemma}).
\end{proof}

\subsection{Concentrating sequences of solutions}

We now prove the generalization of Theorem \ref{concentrating-thm} for concentrating sequences of nearly-maximal solutions:

\begin{theorem}
\label{concentrating-thm4}
Suppose $\omega_n \to \omega$ converge in the Korenblum topology. Then, the associated nearly-maximal solutions of the Gauss curvature equation converge weakly on the unit disk: $u_{\omega_n} \to u_{\omega}$.
\end{theorem}

The proof of Theorem \ref{concentrating-thm4} rests on three simple observations:

\begin{lemma}
\label{jen1}
If $E \subset \mathbb{S}^1$ is a Beurling-Carleson set and $\alpha \ge 1$ then
\begin{equation}
\label{eq:jen1a}
\int_{K_E^\alpha} \frac{|dz|^2}{1-|z|}  \asymp \| E \|_{\BC}.
\end{equation}
If the sets $E_n \to E$ in $\BC$, then
\begin{equation}
\label{eq:jen1b}
\int_{K_{E_n}^\alpha} \frac{|dz|^2}{1-|z|}  \to \int_{K_{E}^\alpha} \frac{|dz|^2}{1-|z|}.
\end{equation}
\end{lemma}
We leave the verification to the reader.

\begin{lemma}
\label{jen2}
 Suppose $u$ is a nearly-maximal solution of the Gauss curvature equation and $\omega \in M_{\BC}(\overline{\mathbb{D}})$. Then, $u = u_\omega$ if and only if
\begin{equation}
\label{eq:jen-sub3}
(u_{\mathbb{D}} - u) (z) = \log \frac{1}{|I_{\omega}(z)|} - \frac{1}{2\pi} \int_{\mathbb{D}} J_u(z,w) \, |dw|^2.\end{equation}
where
$$
J_u(z,w) = \Bigl (e^{2u_{\mathbb{D}}}(w) - e^{2u}(w)\Bigr) \, G(z,w).
$$
\end{lemma}

The lemma follows after applying the  Poisson-Jensen formula for subharmonic functions on $\mathbb{D}_r$ and taking $r \to 1$.

\begin{lemma}
\label{point-to-weak}
Suppose  the functions $h_n: \mathbb{D} \to \mathbb{R}$ converge pointwise to $h$ and are locally uniformly bounded. Then, they converge weakly to $h$, that is, for any test function $\phi \in C_c^\infty(\mathbb{D})$, $\int_{\mathbb{D}} h_n \phi \, |dz|^2 \to \int_{\mathbb{D}} h \phi \, |dz|^2.$
\end{lemma}

This is immediate from the dominated convergence theorem.

\begin{proof}[Proof of Theorem \ref{concentrating-thm4}]
 To prove the theorem, we show that if a sequence of solutions $u_n = u_{\omega_n}$ converges weakly to a solution $u$, then $u = u_\omega$ where $\omega$ is the weak limit of the $\omega_n$.
The reduction described in the proof of Theorem \ref{concentrating-thm} allows us to assume that each measure ${\omega_n}$ is supported on a Korenblum star $K_{E_n}$ with $E_n \to E$ in $\BC$.
The strategy is rather straightforward.
By Lemma \ref{jen2}, for each $n = 1,2, \dots$, we know that
\begin{equation}
\label{eq:jen-sub}
u_{\mathbb{D}}(z) - u_{n} (z) - \log \frac{1}{|I_{\omega_n}(z)|} = - \frac{1}{2\pi} \int_{\mathbb{D}}  J_{u_n}(z,w) \,  |dw|^2.
\end{equation}
We claim that if we take the weak limit of (\ref{eq:jen-sub}) as $n \to \infty$, we will end up with
\begin{equation}
\label{eq:jen-sub2}
u_{\mathbb{D}}(z) - u(z) - \log \frac{1}{|I_{\omega}(z)|} = -  \frac{1}{2\pi} \int_{\mathbb{D}} J_u(z,w) \,  |dw|^2,
\end{equation}
which would mean that $u = u_\omega$.

By assumption, the $u_n$ converge weakly to $u$, while by Lemma \ref{weak-log}, $\log\frac{1}{|I_{\omega_n}(z)|}$ converge weakly to $\log\frac{1}{|I_{\omega}(z)|}$. It remains to show that
\begin{equation}
\label{eq:j-convergence}
\frac{1}{2\pi} \int_{\mathbb{D}}  J_{u_n}(z,w) \,  |dw|^2 \to \frac{1}{2\pi} \int_{\mathbb{D}}  J_{u}(z,w) \,  |dw|^2
\end{equation}
also converge weakly.
For this purpose, we will use the following bounds on the integrands $J_n(z,w) = J_{u_n}(z,w)$:
 \begin{itemize}
 \item
 If $d_{\mathbb{D}}(w, z) \le 1$, we use the bound $J_n(z,w) \le C_1(z) \cdot G(z,w)$. Note that the singularity of the Green's function is integrable.
 \item
For $w \in K^2_{E_n}$ with $d_{\mathbb{D}}(w, z) > 1$, we use the coarse estimate
$$
J_n(z, w) \le e^{2u_{\mathbb{D}}}(w) \, G(z, w) \le C_2(z) \cdot \frac{1}{1-|w|}.
$$
\item
For $w \in \mathbb{D} \setminus K_{E_n}^2$ with $d_{\mathbb{D}}(w, z) > 1$, we use the fine estimate
\begin{align*}
J_n(z, w) & \le C_2(z) \cdot \frac{1}{1-|w|} \cdot \bigl (1 - e^{-2(u_{\mathbb{D}} - u_n)}(w) \bigr ) \\
 & \le C_3(z) \cdot \frac{1}{1-|w|} \cdot \log \frac{1}{|I_{\omega_n}(w)|} \\
  & \le M C_3(z) \cdot \frac{1}{1-|w|} \cdot \exp\bigl(-d_{\mathbb{D}}(w,K_{E_n}^2)\bigr),
\end{align*}
where $M = \sup_{n \ge 1} \omega_n(\overline{\mathbb{D}})$. The second inequality follows from (\ref{eq:fund2}) while the third inequality is provided by Lemma \ref{gammaF-estimate2} (which holds for generalized Blaschke products).
\end{itemize}
In view of the ``area convergence'' (\ref{eq:jen1b}), the second estimate on $J_n$ and the weak convergence of $u_n \to u$ shows
$$
\frac{1}{2\pi} \int_{K_{E_n}^2}  J_{u_n}(z,w) \,  |dw|^2 \, \to \, \frac{1}{2\pi} \int_{K_E^2}  J_{u}(z,w) \,  |dw|^2.
$$
For $0 < \theta_1 < \theta_2 \le 1$, let
$K^2_E(\theta_1, \theta_2)$ denote
 $K^2_E(\theta_2) \setminus K^2_E(\theta_1)$.
A similar argument implies that
$$
\frac{1}{2\pi} \int_{K_{E_n}^2(e^{-(k+1)}, e^{-k})}  J_{u_n}(z,w) \,  |dw|^2 \, \to \, \frac{1}{2\pi} \int_{K_E^2(e^{-(k+1)},e^{-k})}  J_{u}(z,w) \,  |dw|^2,
$$
for any $k \ge 0$.
 However, by the third estimate on $J_n$, these integrals decay exponentially in $k$, which proves the pointwise
convergence in (\ref{eq:j-convergence}).

 Since $C_1(z), C_2(z), C_3(z)$ can be taken to be continuous in $z \in \mathbb{D}$, the functions
$z \to \frac{1}{2\pi} \int_{\mathbb{D}}  J_{u_n}(z,w) \, |dw|^2$ are locally uniformly bounded, which allows us to use
Lemma \ref{point-to-weak} to upgrade pointwise convergence to weak convergence.
This completes the proof.
\end{proof}

\section{Diffuse  sequences}
\label{sec:equidiffuse}

We now turn our attention to diffuse sequences. We show:

\begin{theorem}
\label{equidiffuse-thm}
For any totally diffuse sequence of measures $\{\mu_n\} \subset M_{\BC}(\overline{\mathbb{D}})$ whose
 masses
$\mu_n(\overline{\mathbb{D}}) < M$ are uniformly bounded above, the associated nearly-maximal solutions of the Gauss curvature equation $u_{\mu_n}$ converge weakly to   $u_{\mathbb{D}}$.
\end{theorem}

The proof is similar to that of \cite[Theorem 1.10]{inner}, but requires a slightly more intricate argument since we need to decompose measures supported on the closed unit disk.

 For an arc $I \subset \mathbb{S}^1$ of the unit circle, we write  $$\square_{I, r, R} := \{z \in \overline{\mathbb{D}}: z/|z| \in I, \, r \le |z| \le R \},$$ with the convention that we include the left edge into $I_{r,R}$ but not the right edge.

\subsection{Roberts decompositions}
\label{sec:roberts}

Similarly to the original Roberts decomposition for measures supported on the unit circle \cite{roberts}, our decomposition will depend on two parameters: a real number $c > 0$  and an integer $j_0 \ge 1$. Set $n_j := 2^{2^{(j+j_0)}}$ and $r_j := 1 - 1/n_j$.

\begin{theorem}
\label{roberts-thm}
Given a finite measure $\mu \in M(\overline{\mathbb{D}})$
on the closed unit disk, one can write it as
\begin{equation}
\label{eq:roberts}
\mu = (\mu_2 + \mu_3 + \mu_4 + \dots) + \nu_{\cone}
\end{equation}
where each measure  $\mu_j$, $j \ge 2$, satisfies
\begin{equation}
\label{eq:first-estimate}
\supp \mu_j \subset \square_{\mathbb{S}^1,r_{j-1},1}
\end{equation}
 and
\begin{equation}
\label{eq:second-estimate}
\mu_j(\square_{I, r_{j-1},1}) \le 2(c/n_j) \log n_j, \qquad \forall I \subset \mathbb{S}^1, \quad |I| = 2\pi/n_j;
\end{equation}
while the cone measure $\nu_{\cone}$ is supported on a Korenblum star $K_{E_{\cone}}$ of norm $\|E_{\cone}\|_{\BC} \le N \bigl (c,j_0, \mu(\overline{\mathbb{D}}) \bigr )$.
\end{theorem}

\begin{proof}
We obtain the decomposition by means of an algorithm which sorts out the mass of $\mu$ into various components.
 For each $j = 2, 3, \dots$, we consider  a partition $P_j$ of the unit circle into $n_j$ equal arcs. Since $n_j$ divides $n_{j+1}$, each next partition can be chosen to be a refinement of the previous one.

\medskip

As Step 1 of our algorithm, we move $\mu|_{B(0,r_1)}$ into $\nu_{\cone}$. (We remove this mass from $\mu$.)

In Step $j$, $j=2,3,\dots$, we consider all intervals in the partition $P_j$. Define an interval to be {\em light} if
$\mu(\square_{I, 0, 1}) \le (c/n_j) \log n_j$ and {\em heavy} otherwise. We do one of the following three operations:

\begin{itemize}
\item[L.] If $I$ is light, we move the mass $\mu |_{\square_{I,0,1}}$ into $\mu_j$.

\item[H1.] If $I$ is heavy, we look at the box $\square_{I, r_{j-1}, r_j}$. If $\mu(\square_{I, r_{j-1}, r_j}) \ge (c/n_j) \log n_j$, we move $\mu |_{\square_{I, r_{j-1}, r_j}}$ into $\nu_{\cone}$.

\item[H2.] If  $\mu(\square_{I, r_{j-1}, r_j}) < (c/n_j) \log n_j$, we move $\mu|_{\square_{I, r_{j-1}, r_j}}$ to $\mu_j$. We also move some mass from $\mu|_{\square_{I,r_j,1}}$ to $\mu_j$ so that
$ \mu_j(\square_{I,0,1}) = (c/n_j) \log n_j.$
\end{itemize}
After we followed the above instructions for  $j=2, 3, \dots$, it is possible that the measure $\mu$ has not been exhausted completely: some ``residual'' mass may remain on the unit circle. We move this remaining mass to  $\nu_{\cone}$.

\medskip

From the construction, it is clear that the conditions  (\ref{eq:first-estimate}) and  (\ref{eq:second-estimate}) are satisfied.
The  factor of 2 in (\ref{eq:second-estimate}) is due to the fact that any interval $I \subset \mathbb{S}^1$ of length $2\pi/n_j$ is contained in the union of two adjacent intervals from the partition $P_j$.

Let $\Lambda$ be the collection of light intervals (of any generation) which are maximal with respect to inclusion.
Define $E_{\cone}^* := \mathbb{S}^1 \setminus \bigcup_{I \in \Lambda} \interior I$ as the complement of the interiors of these intervals.  Since the measure $\nu_{\cone}|_{\mathbb{S}^1}$ is supported on the set of points which lie in heavy intervals at every stage, $\supp \nu_{\cone}|_{\mathbb{S}^1} \subset K_{E^*_{\cone}}$.
Observe that if $I$ is an interval of generation $j$,  the box $\square_{I, r_{j-1},r_j}$ is contained in the union of two non-standard Stolz angles emanating from the endpoints of $I$, with a sufficiently wide opening angle. If $I$ is heavy,  these endpoints are contained in $E^*_{\cone}$, from which we see that
$\supp \nu_{\cone}|_{\mathbb{D}}$ is contained in the generalized Korenblum star $K_{E^*_{\cone}}(1/10)$.

 To check that $E^*_{\cone}$ is a Beurling-Carleson set, we follow the computation from Roberts \cite{roberts}.
The relation $\log n_{j+1} = 2 \log n_j$ shows
\begin{equation}
\label{eq:roberts-estimate}
\sum_{I \in \Lambda} |I| \log \frac{1}{|I|} \, \lesssim \,
\sum_{I \in P_2} |I| \log \frac{1}{|I|} + \sum_{\text{heavy}} |J| \log \frac{1}{|J|} \, \lesssim \,
2^{j_0} + \mu(\overline{\mathbb{D}}),
\end{equation}
where we have used the fact that a maximal light interval of generation $j \ge 3$ is contained in a heavy interval of the previous generation.

In order to prove the theorem as stated, we must show that $\nu_{\cone}$ is contained in a genuine Korenblum star $K_{E_{\cone}}$ rather than $K_{E^*_{\cone}}(1/10)$. To achieve this, we note that a heavy box  $\square_{I, r_{j-1},r_j}$ is contained in the union of 10 genuine $90^\circ$ Stolz angles rather than 2 wide Stolz angles. This leads us to define
\begin{equation}
\label{eq:econe-def}
E_{\cone} \ = \ E_{\cone}^* \ \cup \ \bigcup_{[x,x+h] \text{ heavy}} \ \bigcup_{i=1}^8 \ \bigl \{ x + (i/10)h\bigr  \}.
\end{equation}
Using nesting properties of heavy intervals, it is not difficult to show that the set $E_{\cone}$ is closed.
By Lemmas \ref{simple-entropy-lemma2} and \ref{simple-normal-families} below,
\begin{equation}
\label{eq:econe-estimate}
\| E_{\cone}\|_{\BC} \le \| E^*_{\cone}\|_{\BC} + \sum_{\text{heavy}} |J| \log \frac{10}{|J|},
\end{equation}
which means that $E_{\cone}$ is also a Beurling-Carleson set.
\end{proof}

In the proof above, we made use of two elementary lemmas:

\begin{lemma}
\label{simple-entropy-lemma2}
Let $[a,b] \subset \mathbb{S}^1$ be an arc in the unit circle. For a finite subset $E \subset [a,b]$, let $\mathcal I(E)$ be the collection of intervals that make up the complement $[a,b] \setminus E$.
Given two finite subsets $F_1, F_2 \subset [a,b]$,
$$
\sum_{I \in \mathcal I(F_1 \cup F_2)} |I| \log \frac{1}{|I|} \  \le \  \sum_{I \in \mathcal I(F_1)} |I| \log \frac{1}{|I|}+
\sum_{I \in \mathcal I(F_2)} |I| \log \frac{1}{|I|}.
$$
\end{lemma}
\begin{proof}
To prove the lemma, it suffices to construct an injective mapping
$$
i: \  \mathcal I(F_1 \cup F_2) \ \to \ \mathcal I(F_1) \cup \mathcal I(F_2)
$$
such that $J \supseteq I$ whenever  $i(I) = J$. We map $[x,y] \subset \mathcal I(F_1 \cup F_2)$ to an interval in $\mathcal I(F_1)$ containing it if the point $x$ belongs to $F_1 \cup \{a\}$ and map $[x,y]$ to an interval in $\mathcal I(F_2)$ otherwise. It is easy to see that no interval in $ \mathcal I(F_1) \cup \mathcal I(F_2)$ gets used twice.
\end{proof}

\begin{lemma}
\label{simple-normal-families}
Suppose $F_1 \subseteq F_2 \subseteq F_3 \subseteq \dots$ is an increasing sequence of finite subsets of the unit circle such that the norms $\| F_n \|_{\BC} \le M$ are bounded.
Let $F$ be the closure of their union. Then, $F$ is a Beurling-Carleson set with $\| F \|_{\BC} \le M$.
\end{lemma}

For a more general statement, see \cite[Lemma 7.6]{HKZ}.

\subsection{Estimating nearly diffuse solutions}
Since the proof of Theorem \ref{equidiffuse-thm} is very similar to that of \cite[Theorem 1.10]{inner}, we only give a  sketch the argument and refer the reader to \cite[Section 6]{inner} for the details. We will need the following lemma:
\begin{lemma}
\label{comparison}
Suppose $\mu \in M(\overline{\mathbb{D}})$ is a finite measure on the closed unit disk which satisfies
 $
\supp \mu \subset \square_{\mathbb{S}^1, 1-1/n,1}
$
and
$$
 \mu(\square_{I,1-1/n,1}) \le 2c \cdot |I| \log \frac{1}{|I|}, \qquad \forall I \subset \mathbb{S}^1, \quad |I| = 2\pi/n.
$$
Then,
$$
|I_\mu(z)| > \frac{1}{(1-|z|^2)^{c'}}, \qquad |z| < 1 - 2/n,
$$
for some $c' \asymp c$.
\end{lemma}

The lemma is well known when $\supp \mu \subseteq \mathbb{S}^1$, e.g.~see \cite[Lemma 2.2]{roberts}.
The general case is similar.
In this section, we fix the parameter $c > 0$ in the Roberts decomposition so that $c' < 1/10$, i.e.~
\begin{equation}
\label{eq:third-estimate}
  |I_{\mu_j}(z)| > \frac{1}{(1-|z|^2)^{1/10}}, \qquad \text{for } z \in \square_{\mathbb{S}^1, 0, r_{j-2}}.
\end{equation}
For $0 < r \le 1$, $C > 0$, let $u_{r,C}$ denote the unique solution of $\Delta u = 4e^{2u}$ defined on
$\mathbb{D}_r$ with constant boundary values $u|_{\partial \mathbb{D}_r} \equiv C$. Since $u_{r,C}$ is unique, it must be  radially-invariant. For an explicit formula, see \cite[Lemma 5.10]{inner}. Here, we only mention that for a fixed $z \in \mathbb{D}$,
\begin{equation}
\label{eq:tends-to-infinity}
\lim_{r \to 1^-, \, C \to \infty} u_{r,C}(z) = u_{\mathbb{D}}(z),
\end{equation}
uniformly on compact subsets of the disk \cite[Corollary 5.11]{inner}.

To prove Theorem \ref{equidiffuse-thm}, we need to show that if a  measure $\mu \in M_{\BC}(\overline{\mathbb{D}})$   gives little mass to any Korenblum star $K_E$ of norm $\le N$, then $u_\mu$ is close to $u_{\mathbb{D}}$ in the weak topology.
Consider the Roberts decomposition
$$
 \mu = (\mu_2 + \mu_3 + \mu_4 + \dots) + \nu_{\cone}
$$
 with parameter $j_0 \ge 1$ large.
For any $\delta > 0$, by asking for $N = N \bigl (j_0, \delta, \mu(\overline{\mathbb{D}}) \bigr) > 0$ to be sufficiently large, we can guarantee that
$\nu_{\cone}(\overline{\mathbb{D}}) \le \delta$.
From Lemmas \ref{weak-log} and \ref{fundamental-lemma4}, it follows that the impact of $\nu_{\cone}$ on $u_{\mu}$ is not significant:
\begin{equation}
\label{eq:u-mu}
 u_\mu \, = \, \Lambda^{\tilde \mu} \biggl [u_{\mathbb{D}} -
 \log \frac{1}{|I_\mu|} \biggr ]
 \, \approx \,
 \Lambda^{\tilde \mu_2 +  \tilde \mu_3 + \dots + \dots} \biggl [u_{\mathbb{D}} - \log \frac{1}{ |I_{\mu_2 + \mu_3 + \dots}|}  \biggr ].
\end{equation}
Thus Theorem \ref{equidiffuse-thm} reduces to showing:
\begin{theorem}
\label{equidiffuse-thm2}
Suppose $\mu \in M_{\BC}(\overline{\mathbb{D}})$ is a finite measure on the closed unit disk which can be expressed as a countable sum
 $$\mu = \mu_2 + \mu_3 + \mu_4 + \dots,$$ where each piece satisfies (\ref{eq:first-estimate}) and (\ref{eq:second-estimate}) and $c > 0$ is sufficiently small so that (\ref{eq:third-estimate}) holds. Then $u_{\mu} > u_{r_0, {(4/5)u_\mathbb{D}}}$ on $\mathbb{D}_{r_0}$.
\end{theorem}

\begin{proof}[Sketch of proof]
To simplify notation, let us write $\Lambda_r := \Lambda_r^0$. By Collorary \ref{lambda-sequences}, it suffices to prove the theorem when $\mu = \mu_2 + \mu_3 + \mu_4 + \dots + \mu_j$ is a finite sum.
Consider the non-singular
solution of the Gauss curvature equation
\begin{equation}
 \label{eq:clever-strategy}
\tilde u \, := \, \Lambda_{r_{0}} \Biggl [  \dots \Lambda_{r_{j-3}} \biggl [  \Lambda_{r_{j-2}}
 \biggl [  u_{\mathbb{D}} - \log \frac{1}{|I_{\mu_j}|} \biggr ] - \frac{1}{ |I_{\mu_{j-1}}|} \biggr ] \dots - \log \frac{1}{|I_{\mu_2}|} \Biggr ]
\end{equation}
defined on the disk $\mathbb{D}_{r_0}$. By the monotonicity properties of $\Lambda$ and the repeated use of
 Lemma \ref{fundamental-lemma4}, we have
\begin{align*}
\tilde u & \, \le \, \Lambda^{\tilde \mu_2 +  \tilde \mu_3 + \dots + \tilde \mu_j} \Biggl [  \dots \Lambda^{\tilde \mu_{j-1} + \tilde \mu_j} \biggl [  \Lambda^{\tilde \mu_j}
 \biggl [  u_{\mathbb{D}} - \log \frac{1}{|I_{\mu_j}|} \biggr ] - \frac{1}{ |I_{\mu_{j-1}}|} \biggr ] \dots - \log \frac{1}{|I_{\mu_2}|} \Biggr ], \\
& \, = \, \Lambda^{\tilde \mu_2 +  \tilde \mu_3 + \dots + \tilde \mu_j} \biggl [u_{\mathbb{D}} - \log \frac{1}{ |I_{\mu_2 + \mu_3 + \dots + \mu_j}|}  \biggr ], \\
& = u_\mu
\end{align*}
on $\mathbb{D}_{r_0}$, where we made use of the fact that $\supp \mu_j \cap \mathbb{D}_{r_{j-2}} = \emptyset$. To show that $\tilde u > u_{r_0, u_{\mathbb{D}} - \log 2}$ on $\mathbb{D}_{r_0}$,
we  estimate  $\tilde u$ by recursively unwinding the definition (\ref{eq:clever-strategy}):
\begin{itemize}
\item[0.]
We begin with $u_{\mathbb{D}} - \log 2$.

\item[1.]  We subtract $\log \frac{1}{|I_{\mu_{j}}|}$. By the estimate
(\ref{eq:third-estimate}), $$u_{\mathbb{D}} - \log 2 - \log \frac{1}{|I_{\mu_j}|} \ge  (4/5)  u_{\mathbb{D}}, \qquad \text{on } \partial \mathbb{D}_{r_{j-2}}.$$

\item[2.] We form the solution $u_{r_{j-2}, (4/5)  u_{\mathbb{D}}}$. By the computation in \cite[Section 6]{inner},  $u_{r_{j-2}, (4/5)  u_{\mathbb{D}}} > u_{\mathbb{D}} - \log 2$ on $\partial \mathbb{D}_{r_{j-3}}$.
\end{itemize}
Repeating this process gives the desired estimate.
\end{proof}

\subsection{Diffuse sequences lose mass}
\label{sec:diffuse-lose-mass}

To complete the proof of Theorem \ref{main-thm4}, we need to show that if a sequence of measures $\mu_n \to \mu$ does not converge in the Korenblum topology, then the associated nearly-maximal solutions
 $u_{\mu_n}$ do not converge weakly to $u_\mu$.
We first make a simple observation:

\begin{lemma}
\label{some-mass-is-lost}
Suppose $\mu_n \to \mu$ and $\nu_n \to \nu$ are two weakly convergent sequences of measures in $M_{\BC}(\overline{\mathbb{D}})$.
If $u_{\mu_n + \nu_n} \to u_{\mu+\nu}$ then $u_{\mu_n} \to u_{\mu}$.
\end{lemma}

\begin{proof}
Passing to a subsequence, we may assume that $u_{\mu_n}$ converges weakly to a nearly-maximal solution $u_{\mu^*}$ with $\mu^* \le \mu$.
By Lemma \ref{fundamental-lemma4},
$$ (u_{\mathbb{D}} - u_{\mu_n + \nu_n}) - (u_{\mathbb{D}} - u_{\mu_n}) \, = \, u_{\mu_n} - u_{\mu_n + \nu_n}\, \le \, \log \frac{1}{|I_{\nu_n}|}.$$
Taking $n \to \infty$ and examining the boundary data, we arrive at $(\mu + \nu) - \mu^* \le \nu$
which implies that $\mu^* = \mu$.
\end{proof}

Suppose $\mu_n \to \mu$ is a sequence of measures and
$$
\mu_n =  {\overline \mu}_n + \nu_{n, \cone} =  (\mu_{n,2} + \mu_{n,3} + \mu_{n,4} + \dots) + \nu_{n, \cone},
$$
are their Roberts decompositions (with the same parameters $c, j_0$). Recall that $E_{n, \cone}^* := \mathbb{S}^1 \setminus \bigcup_{I \in \Lambda} \interior I$,
where $\Lambda$ is the collection of light intervals which are maximal with respect to inclusion. Take
the threshold $\eta = 1/(2n_2)$.
 Since the maximal light intervals in $P_2$ are longer than $\eta$,
\begin{equation}
\label{eq:econe1}
\| E_{n, \cone}^* \|_{\BC_\eta} \, = \, \sum_{I \in \Lambda, \, |I| < \eta} |I| \log \frac{1}{|I|} \, \lesssim \,
 \sum_{\text{heavy}} |J| \log \frac{1}{|J|} \, \lesssim \,
 \frac{1}{c} \cdot \mu(\overline{\mathbb{D}}).
\end{equation}
Recall that $\nu_{n,\cone}$ is supported on a Korenblum star $K_{E_{n, \cone}}$ where $E_{n, \cone}$ is given
  by (\ref{eq:econe-def}). Its local entropy is only slightly larger:
  \begin{equation}
\label{eq:econe2}
\| E_{n,\cone}\|_{\BC_\eta} \, \le \, \| E^*_{n,\cone}\|_{\BC_\eta} + \sum_{\text{heavy}} |J| \log \frac{10}{|J|} \, \lesssim \,
 \frac{1}{c} \cdot \mu(\overline{\mathbb{D}}).
\end{equation}

If $\{ \mu_n \}$ does not converge in the Korenblum topology, then the diffuse part of $\{\mu_n\}$ is non-trivial, that is, one may decompose
$\mu_n = \nu_n + \tau_n$ such that $\nu_n \to \nu$ is concentrating, $\tau_n \to \tau$ is diffuse and $\tau \ne 0$.
By asking for the parameter $c = C > 0$ to be large, we can make the right hand side of (\ref{eq:econe2}) as
 small as we wish.
From the definition of
a diffuse sequence, it follows that for $n$ large, a definite chunk of $\tau_n$ must fall into $\overline{\mu}_n$.
If $c$ is the constant from Theorem \ref{equidiffuse-thm2}, then $u_{(c/C)  \overline{\mu}_n} \to u_{\mathbb{D}}$.
By Lemma \ref{some-mass-is-lost}, $u_{\mu_n}$ cannot converge to $u_{\mu}$.

 \subsection{An instructive example}

For a fixed $M > 0$, consider the sequence of measures $\mu_{n,M} = \sum_{k=1}^n\delta_{e^{ik\theta_n}}/n$,
where $\theta_n$ is chosen so that $n \theta_n \log \frac{1}{\theta_n} = M$.
We show:
\begin{lemma}
For  $M > 0$  sufficiently large, the nearly-maximal solutions $u_{n,M} = u_{\mu_{n,M}}$ converge to $u_{\mathbb{D}}$.
For $M > 0$ sufficiently small, the $u_{n,M}$ do not converge to $u_{\mathbb{D}}$.
\end{lemma}

\begin{proof}
For any arc $I \subset \mathbb{S}^1$ of length $\theta_n/2$,
$$\mu_{n,M}(I) \le 1/n = (1/M) \cdot \theta_n \log\frac{1}{\theta_n} \le (3/M)  \cdot |I| \log\frac{1}{|I|}.$$
For the first assertion, it is enough to request that $3/M < c$ where $c$ is the constant from
 Theorem \ref {equidiffuse-thm2}.
If the second assertion were false, a diagonalization argument would produce a sequence $u_{n_j, M_j} \to u_{\mathbb{D}}$ with $n_j \to \infty$ and $M_j \to 0$. But this diagonal sequence is concentrating, so by Theorem \ref{concentrating-thm4}, its limit should be $u_{\delta_1}$, which is a contradiction.
\end{proof}

\begin{corollary}
There exists a diffuse sequence of measures $\mu_n \in M_{\BC}(\overline{\mathbb{D}})$ such that  $u_{\mu_n}$ do not converge to 
$u_{\mathbb{D}}$ but $u_{k \cdot \mu_n} \to u_{\mathbb{D}}$ for some $k < 1$.
\end{corollary}

\begin{proof}
By the second statement of the lemma, we can choose
 $M > 0$ so that the $u_{\mu_{n,M}}$ do not converge to  $u_{\mathbb{D}}$.
If $k = (M/3)c$, then
for any arc  $I \subset \mathbb{S}^1$ of length $\theta_n/2$, $k \cdot \mu_{n,M}(I) \le c \cdot |I|\log\frac{1}{|I|}$, which implies
 that $u_{k \cdot \mu_{n,M}}$ tends to $u_{\mathbb{D}}$ as $n \to \infty$.
\end{proof}

\begin{remark}
Let $k_0 = \sup \{k > 0 : u_{k \cdot \mu_n} \to u_{\mathbb{D}} \text{ as }n\to\infty \}$. One can show that
$u_{k \cdot \mu_n} \to u_{(k-k_0) \cdot \mu_n}$ for $k \ge k_0$. For the $\le$ direction, one may write
$$(u_{\mathbb{D}} - u_{k \cdot \mu_n}) - (u_{\mathbb{D}} - u_{k_0 \cdot \mu_n}) \le  \log \frac{1}{|I_{(k-k_0) \cdot \mu_n}|},
$$
take $n \to \infty$ and examine the boundary data like in the proof of Lemma \ref{some-mass-is-lost}.
The $\ge$ direction is harder and relies on the Solynin-type inequality $ u_{\mu} + u_{\nu} \ge u_{\mu + \nu} + u_{\mathbb{D}}$ for
any  $\mu, \nu \in M_{\BC}(\overline{\mathbb{D}})$. While we will not prove Solynin's inequality here, we can refer the reader to
\cite[Equation 5.1]{inner} for a special case.
\end{remark}

\section{Invariant subspaces of Bergman spaces}

For a fixed $\alpha > -1$ and $1 \le p < \infty$, consider the weighted Bergman space $A_\alpha ^p(\mathbb{D})$ of holomorphic functions satisfying the norm boundedness condition
(\ref{eq:beurling-space-def}).
Let $\{I_n\}$ be a sequence of inner functions which converge uniformly on compact subsets of the disk to an inner function $I$. Assume that the measures $\mu(I_n)$ and $\mu(I)$ are in $M_{\BC}(\overline{\mathbb{D}})$.
Let $[I_n] \subset A^p_\alpha$ be the $z$-invariant subspace generated by $I_n$. In this section, we prove Theorem \ref{main-thm3} which says that
 $\lim_{n \to \infty} [I_n] = [I]$ if and only if the measures $\mu(I_n)$ converge to $\mu(I)$ in the Korenblum topology.

In general, one has semicontinuity in one direction:
\begin{equation}
\label{eq:automatic-inclusion}
[I] \subseteq \liminf_{n \to \infty} \, [I_n].
\end{equation}
To see this, note that if $f \in [I]$, then it may be approximated in norm by $p_k I$ for some polynomials $\{p_k\}_{k=1}^\infty$. Diagonalization allows us to express $f$ as the limit of $p_{k(n)} I_{n} \in [I_n]$.

\subsection{Concentrating sequences: special case}

\begin{theorem}
\label{bergman-concentration}
 Suppose $I_n \to I$ is a sequence of inner functions which converges uniformly on compact subsets of the unit disk. If the zero structure of $I_n$ belongs to a Korenblum star $K_{E_n}$ and the $E_n \to E$ converge in $\BC$ then $[I_n] \to [I]$.
\end{theorem}

The above theorem is essentially due to Korenblum \cite{korenblum}, although our version has some extra uniformity.
  For a Beurling-Carleson set $E$, one can construct an outer function $\Phi_E(z) \in C^\infty(\overline{\mathbb{D}})$  which vanishes precisely on $E$ and does so to infinite order.
  Examining the construction in \cite[Proposition 7.11]{HKZ}, we may assume that $\Phi_E$ enjoys two nice properties:

 \begin{enumerate}
 \item
  The function $\Phi_E(z)$ varies continuously with the Beurling-Carleson set $E$, in the sense that
   $\Phi_{E_n} \to \Phi_E$ uniformly on compact subsets of the disk if $E_n \to E$ and  $\|E_n\|_{\BC} \to \|E\|_{\BC}$.

 \item
For each $N \ge 0$,
$$
|\Phi_E(z)| \cdot \dist(z, E)^{-N} \le C_E(N)
$$ is bounded by a constant which depends continuously on $E$. It is convenient to take $C_0 = 1$ so that
$|\Phi_E(z)| \le 1$ on the disk.
 \end{enumerate}

A brief sketch of the construction will be provided in Appendix \ref{sec:carleson-outer}.
 The central idea in Korenblum's vision is the following division principle:
 \begin{theorem}[Korenblum's division principle]
 \label{korenblum-division}
 Suppose $I$ is an inner function with zero structure $\supp \mu(I) \subset K_E$ and $f \in [I]$. For any $\delta > 0$,
\begin{equation}
 \label{eq:korenblum-division}
f^\delta(z) := (\Phi^\delta_E/I) f(z) \in A^p_\alpha,
 \end{equation}
 with the norm estimate $\|f^\delta\|_{A^p_\alpha} \le C_E \| f \|_{A^p_\alpha}$ where $C_E = C_E \bigl ( \delta, \mu(I)(\overline{\mathbb{D}}) \bigr )$ depends continuously  on $E$, $\delta$ and $\mu(I)(\overline{\mathbb{D}})$.
 \end{theorem}

Assuming Theorem \ref{korenblum-division}, the proof of Theorem \ref{bergman-concentration} runs as follows:

\begin{proof}[Proof of Theorem \ref{bergman-concentration}]

Suppose that a sequence of functions $f_n \in [I_n]$ converges to $f$ in $A^p_\alpha$. Norm convergence implies that the $f_n$ converge to $f$ uniformly on compact subsets of $\mathbb{D}$. By Korenblum's division principle, for a fixed $\delta > 0$, the functions
  $g_n = (\Phi_{n}^\delta/I_n) \cdot f_n(z)$ have bounded
  $A^p_\alpha$ norms  and converge uniformly on compact subsets to
  $$g = (\Phi^\delta/I) \cdot f(z).$$ Fatou's lemma implies that $g \in A^p_\alpha$ and therefore $\Phi^\delta \cdot f = I g \in [I]$.
Taking $\delta \to 0$ shows that $f \in [I]$ and therefore
$[I] \supseteq \limsup_{n \to \infty} \, [I_n]$. By (\ref{eq:automatic-inclusion}), the other inclusion is automatic.
\end{proof}

Since the exact statement of  Theorem \ref{korenblum-division} is not present in Korenblum's work \cite{korenblum}, we give a proof below.

\begin{proof}[Proof of Korenblum's division principle (Theorem \ref{korenblum-division})]
We first consider the case when $I$ is a finite Blaschke product and $E$ is a finite set. Afterwards, we will deduce the general case by a limiting argument.
If $I$ is a finite Blaschke product, it is clear that $f^\delta \in A^p_\alpha$. We need to give a uniform estimate on its norm.

Recall that $K_E^2$ denotes the generalized Korenblum star of order 2, see (\ref{eq:gen-korstar}) for the definition.
According to Lemma \ref{gammaF-estimate2},  $|1/I(z)| \le C \bigl (\mu(I)(\overline{\mathbb{D}}) \bigr )$ is uniformly bounded on $\mathbb{D} \setminus K_E^2$ so that
$|f^\delta(z)| \le C |f(z)|$ there.

To estimate $f^\delta$ on $K_E^2$, we examine its values on the boundary $\partial K_E^2$. It is well known that a  function in Bergman space does not grow too rapidly:
\begin{equation}
\label{eq:simple-bergman-bound}
|f(z)| \le C_2\|f\|_{A^p_\alpha}(1-|z|)^{-\beta}, \qquad z \in \mathbb{D},
\end{equation}
for some $\beta = \beta(p,\alpha) > 0$.
However, the $C^\infty$ decay of the outer function $\Phi^\delta$ cancels out this grows rate on $\partial K_E^2$ and we end up with
$$
|f^\delta(z)| \le C_3 \| f \|_{A^p_\alpha}, \qquad z \in \partial K_E^2.
$$
Since $f^\delta  \in A^p_\alpha$, we can use the Phragm\'en-Lindel\"of principle to conclude that this bound extends to the interior of $K_E^2$. Putting the above estimates together completes the proof when $I$ is a finite Blaschke product.

For the general case, we approximate $I$ uniformly on compact subsets by finite Blaschke products $I_n$ whose zeros are contained in $K_E \supset \mu(I)$. Using the semicontinuity property (\ref{eq:automatic-inclusion}), we may then approximate $f \in [I]$ by $f_n \in [I_n]$ in the $A^p_\alpha$-norm.
By the finite case of the lemma, $f^\delta_n = (\Phi^\delta_E/I_n) f_n(z) \in {A^p_\alpha}$ with $\|f^\delta_n\|_{A^p_\alpha}$ bounded above. By Fatou's lemma, $\|f^\delta\|_{A^p_\alpha} \le \liminf_{n \to \infty} \|f^\delta_n\|_{A^p_\alpha}$ as desired.
\end{proof}

\subsection{Concentrating sequences: general case}

Suppose $I$ is an inner function with $\mu(I) \in M_{\BC}(\overline{\mathbb{D}})$. Suppose $I^N \to I$ is an approximating  sequence of inner functions such that $\mu(I^N) \le \mu(I)$ is supported on a Korenblum star of norm $\le N$. We claim that $[I^N] \to [I]$. The inclusion $\liminf_{N \to \infty} [I^N] \supseteq [I]$ is trivial. Conversely, given a sequence $f^N \in [I^N]$ converging to $f$,  the sequence $f^N(I/I_N) \in [I]$ will also converge to $f$.
Since $[I]$ is closed,  $f \in [I]$ and $\limsup_{N \to \infty} [I^N] \subseteq [I]$, which proves the claim.

\begin{theorem}
Suppose $I_n \to I$ is a  sequence of inner functions which converges uniformly on compact subsets of the disk. If the associated measures $\mu(I_n) \in M_{\BC}(\overline{\mathbb{D}})$ converge in the Korenblum topology, then $[I_n] \to [I]$.
\end{theorem}

\begin{proof}
By the definition of the Korenblum topology, there exist ``approximations''
$I_n^N \to I^N$  supported on  Korenblum stars of norm $\le N$. By Thereom \ref{bergman-concentration},
$$
\limsup_{n \to \infty}\, [I_n] \, \subseteq \, \limsup_{n \to \infty}\, [I_n^N] \, = \, [I^N] \, \to_{N \to \infty} \, [I].
$$
The other inclusion follows from (\ref{eq:automatic-inclusion}).
\end{proof}

\subsection{Diffuse sequences}

\begin{theorem}
\label{bergman-diffuse}
Suppose $I_n \to I$ is a convergent sequence of inner functions such that the associated measures $\mu(I_n)$ are totally diffuse. Then, $[I_n] \to [1]$.
\end{theorem}

To prove the above theorem, we closely follow the work of Roberts \cite{roberts}.
Suppose $\mu \in M(\overline{\mathbb{D}})$ is a measure on the closed unit disk which is very close to being diffuse, that is, gives $\le \varepsilon$ mass to any Korenblum star of order $\le N$. We need to show that the
distance $d(1, [I_\mu])$ from the $z$-invariant subspace $[I_\mu] \subset  A^p_\alpha$ to the constant function 1 is small.
We consider the $(c,j_0)$ Roberts decomposition
$$
\mu = (\mu_2 + \mu_3 + \mu_4 + \dots) + \nu_{\cone}
$$
from Section \ref{sec:equidiffuse},
where the parameter $c$ is small and $j_0$ is large. If $N(c, j_0)$ is large, then the assumption on $\mu$ guarantees that
 $\nu_{\cone}(\overline{\mathbb{D}}) \le \varepsilon$.
Set $\overline{\mu} = \mu_2 + \mu_3 + \mu_4 + \dots$.

We may instead show that $d(1, [I_{\overline{\mu}}])$ is small since  the triangle inequality would imply that $d(1, [I_\mu])$ is also small.
In \cite{roberts}, Roberts proved such an estimate for singular inner functions (in which case, the measures $\mu_j$ are supported on the unit circle).
 Roberts' argument is a clever iterative scheme which is quite similar to the one employed in
  Section \ref{sec:equidiffuse}. Actually, the techniques of Section \ref{sec:equidiffuse} are adapted from Roberts' work where  the use of the corona theorem is replaced with estimates on solutions of the Gauss curvature equation.
 In our setting, the function $I_\mu$ might have zeros and therefore $\mu_j$ are measures on the closed unit disk.
 Nevertheless, Roberts' argument (\cite[Lemmas 2.3 and 2.4]{roberts}) extends to this more general case almost verbatim.

\begin{lemma}
[cf.~Lemma 2.3 of \cite{roberts}]
Fix  $\beta>0$ so that $\|z^n\|_{A^p_\alpha} \le n^{-\beta}$ for $n \ge 2$.
Suppose $I$ is an inner function which enjoys the estimate
\begin{equation}
\label{eq:roberts-argument}
|I(z)| \ge n^{-\gamma}, \qquad |z| \le 1 - 1/n.
\end{equation}
If $0 < \gamma < (\beta/3)K$ where $K$ is the constant from the corona theorem and $n \ge N(\gamma)$ is sufficiently large, then there exists a function $g \in H^\infty(\mathbb{D})$ with
\begin{equation}
\|g \|_\infty \le n^{\beta/3}, \qquad
\|1 - gI \|_{A^p_\alpha} \le n^{-2\beta/3}.
\end{equation}
\end{lemma}

Roberts introduced the function $D[\{n_1, n_2, \dots, n_k\}]$ which is defined recursively by $D[\emptyset] = 0$
and
$D[\{n_1, n_2, \dots, n_k\}] = n_1^{\beta/3}  D[\{n_2, n_3, \dots, n_k\}] + n_1^{-2\beta/3}.$
 In view of monotonicity, this definition naturally extends to infinite sequences.

\begin{lemma}
[cf.~Lemma 2.4 of \cite{roberts}]
  Suppose
   $I_{0}, I_{1}, \dots, I_{k-1}$ are inner functions such that
\begin{equation}
\label{eq:roberts-argument2}
|I_{j}(z)| \ge n_j^{-\gamma}, \qquad |z| \le 1 - 1/{n_j}, \qquad  j=0,1,\dots,k-1.
\end{equation}
Assume that $\min(n_0, n_1, \dots, n_{k-1}) \ge N(\gamma)$.
If $I = \prod_{j=0}^{k-1} I_{j}$ then $d(1,[I]) \le D[\{n_j\}]$.
\end{lemma}

Roberts noticed that if $j_0(\beta) \ge 1$ is large, then the sequence of integers $n_j = 2^{2^{(j+j_0)}}$ in the Roberts decomposition (Theorem \ref{roberts-thm}) is sufficiently sparse to ensure that $D[\{n_j\}]$ is small.
By Lemma  \ref{comparison},  $I_{n_j} := I_{\mu_{j+2}}$ verifies the condition
(\ref{eq:roberts-argument2}) for some $\gamma > 0$, which shows that $d(1,[I_{\overline{\mu}}])$ is small.
 This completes our sketch of
Theorem \ref{bergman-diffuse}. We leave the details to the reader.

\begin{remark}
In the special case of the weighted Bergman space $A^2_1$, we can give an alternative
argument based on the methods of this paper. By the Korenblum-Roberts theorem, we may assume that $\mu(I_n) \in M_{\BC}(\overline{\mathbb{D}})$. For each $I_n$, we may
form an inner function $F_n$ with $F_n(0) = 0$ and $\inn F_n' = I_n$. According to Theorem \ref{equidiffuse-thm}, $F_n \to z$
 uniformly on compact subsets. However, the bound $\|F_n\|_{H^\infty} \le 1$ implies that $\|F_n\|_{H^2} \le \|z\|_{H^2}$ which forces $F_n \to z$ to converge in the $H^2$-norm. The Littlewood-Paley formula
$$
\| F_n \|_{H^2} = \frac{1}{\pi} \int_{\mathbb{D}} |F'_n|^2 \log \frac{1}{|z|^2} \, |dz|^2 \asymp \|F'_n\|_{A^2_1}
$$
then shows that $F'_n \to 1$ in the $A^2_1$-norm.
Since $F'_n \in [I_n]$,
$$
\lim_{n \to \infty} [I_n] \supset \lim_{n \to \infty} [F'_n] \supset [1] = A^2_1.
$$
\end{remark}

To complete the proof of Theorem \ref{main-thm3}, we need to show that if $I_{\mu_n} \to I_\mu$ is a sequence
of inner functions whose zero structures $\mu(I_n) \in M_{\BC}(\overline{\mathbb{D}})$ do not converge in the Korenblum topology
to $\mu(I) \in M_{\BC}(\overline{\mathbb{D}})$, then the invariant subspaces $[I_{\mu_n}]$ do not converge to $[I_\mu]$. In view of the strategy outlined in Section \ref{sec:diffuse-lose-mass}, we only need to show:

\begin{lemma}
Suppose $I_{\tau_n} \to I_{\tau}$ and $I_{\nu_n} \to I_{\nu}$ are two sequences of inner functions with
$\tau \in M_{\BC}(\overline{\mathbb{D}})$.
If $[I_{\tau_n}]$ do not converge to $[I_\tau]$, then $[I_{\nu_n + \tau_n}]$ cannot converge to $[I_{\nu+\tau}]$.
\end{lemma}

\begin{proof}
Passing to a subsequence, we may assume that $[I_{\tau_n}] \to [I_{\tau^*}]$ for some $\tau^* < \tau$.
This means that there exists a sequence of polynomials $\{p_n\}$ such that $I_{\tau_n} p_n \to I_{\tau^*}$ in $A^p_\alpha$. Since $I_{\nu_n + \tau_n}p_n \to I_{\nu+\tau^*}$ in $A^p_\alpha$,  $\liminf_{n \to \infty} [I_{\nu_n + \tau_n}] \supseteq [I_{\nu+ \tau^*}]$. However, $[I_{\nu+ \tau^*}] \supsetneq [I_{\nu+ \tau}]$ as $\tau \in M_{\BC}(\overline{\mathbb{D}})$. The proof is complete.
\end{proof}

\appendix
\section[Appendix A. Entropy of universal covering maps]{Entropy of universal covering maps}

Let $m$ be the Lebesgue measure on the unit circle, normalized to have unit mass.
It is well known that if $F$ is an inner function with $F(0) = 0$, then $m$ is $F$-invariant, i.e.~$m(E) = m(F^{-1}(E))$ for any measurable set $E \subset \mathbb{S}^1$. In the work  \cite{craizer}, M.~Craizer showed that if $F \in \mathscr J$, then the integral $$\int_{|z|=1} \log|F'(z)| dm$$ has the dynamical interpretation as the measure-theoretic entropy of $m$. It is therefore of interest to compute it in special cases.
For finite Blaschke products, one may easily compute the entropy using Jensen's formula:
\begin{theorem}
Suppose $F$ is a finite Blaschke product with $F(0)=0$ and $F'(0) \ne 0$. We have
\begin{equation}
\label{eq:maxjensen}
\frac{1}{2\pi} \int_{|z|=1} \log|F'(z)| d\theta \, = \, \sum_{\crit} \log \frac{1}{|c_i|} -  \sum_{\zeros} \log \frac{1}{|z_i|},
\end{equation}
where in the sum over the zeros of $F$, we omit the trivial zero at the origin.
\end{theorem}

In this appendix, we discuss a complementary example:

\begin{theorem}[Pommerenke]
\label{pomm}
Let $P$ be a relatively closed subset of the unit disk not containing 0. Let $\mathcal U_{P}: \mathbb{D} \to \mathbb{D} \setminus P$ be the universal covering map, normalized so that $\mathcal U_P(0) = 0$ and $\mathcal U'_P(0) >  0$. Then $\mathcal U_P \in \mathscr J$ if and only if $P$ is a Blaschke sequence, in which case
\begin{equation}
\label{eq:pomm}
\frac{1}{2\pi} \int_{|z|=1} \log|\mathcal U'_P(z)| d\theta \, = \, \sum_{p_i \in P} \log \frac{1}{|p_i|} -  \sum_{\zeros} \log \frac{1}{|z_i|}.
\end{equation}
\end{theorem}

A theorem of Frostman says that $\mathcal U_P$ is an inner function if and only if the set $P$ has logarithmic capacity 0, see \cite[Chapter 2.8]{collingwood-lohwater}. In particular, $\mathcal U_P$ is inner if
$P$ is countable.

For brevity, we will write $F = \mathcal U_P$.
While Pommerenke did not explicitly state (\ref{eq:pomm}),  in the work \cite{pommerenke}, he proved
the equivalent statement
\begin{equation}
\inn F'(z) \, = \, \prod_{i=1}^k F_{p_i}(z) \, = \,  \prod_{i=1}^k  \frac{F(z)- p_i}{1-\overline{p_i}F(z)},
\end{equation} so we feel that it is appropriate to name the above theorem after him.
Actually, Pommerenke worked in the significantly greater generality of Green's functions for Fuchsian groups of Widom type, so this is only a special case of his result.
Below, we give a  direct proof of Theorem \ref{pomm} which may be of independent interest.

 \subsection{Preliminaries}

We first recall a well known property of Nevanlinna averages:
 \begin{lemma}
 \label{nevanlinna-gap}
If $f \in \mathcal N$ is a function in the Nevanlinna class and is not identically 0, then
 \begin{equation}
\label{eq:nevanlinna-gap}
 \frac{1}{2\pi} \int_{|z|=1}  \log|f(z)| d\theta -
\lim_{r \to 1} \biggl \{ \frac{1}{2\pi} \int_{|z|=r}  \log|f(z)| d\theta \biggr \} = \sigma(f)(\mathbb{S}^1).
\end{equation}
 \end{lemma}

See \cite[Section 3]{inner} for a proof.
For $x \in \mathbb{D}$, let $F_x = T_x \circ F$ denote the  {\em Frostman shift} of $F$ with respect to $x$, where $T_x(z) = \frac{z - x}{1 - \overline{x}z}$.
 Frostman showed that if $x$ avoids an exceptional set $\mathscr E$ of capacity zero, then $F_x$ is a Blaschke product, in which case  $\sigma(F_x) = 0$.
 We will also need:
\begin{lemma}
Let $F$ be an inner function with $F(0) = 0$. For any $x \in \mathbb{D} \setminus \{0\}$,
\begin{equation}
\label{eq:ncf}
\log\frac{1}{|x|} \, = \, \sum_{F(y) = x} \log\frac{1}{|y|} + \sigma(F_x).
\end{equation}
\end{lemma}

\begin{proof}
Taking $f = F_x$ in Lemma \ref{nevanlinna-gap} gives
$$
0 = \lim_{r \to 1} \frac{1}{2\pi} \int_{|z|=r} \log|F_x(z)| d\theta + \sigma(F_x).
$$
The lemma follows after applying Jensen's formula and taking $r \to 1$.
\end{proof}

In the case when $F \in \mathscr J$, Ahern and Clark \cite{ahern-clark} observed that the exceptional set $\mathscr E$ of $F$ is at most countable and that the singular masses of different Frostman shifts $F_x$ are mutually singular. More precisely, they showed that the measure $\sigma(F_x)$ is supported on the set of points on the unit circle at which the radial limit of $F$ is $x$.
Since the singular inner function $\sing F_x$ divides $F_x'$, it must also divide its inner part $\inn F_x' = \inn F'$. This shows that
\begin{equation}
\label{eq:sing-comparison}
\sigma(F') \ge \sum_{x \in \mathscr E} \sigma(F_x).
\end{equation}
In other words, $\inn F'$ is divisible by the product $\prod_{x \in \mathscr E} \sing F_x$.

\subsection{Proof of Theorem A.2 when $P$ is a finite set}

We first prove Theorem \ref{pomm} when $P = \{p_1, p_2, \dots, p_k\}$ is a finite set.
In the formula (\ref{eq:maxjensen}), one considers the sum $\sum_{\crit} \log \frac{1}{|c_i|}$ over critical points. It appears that the identity (\ref{eq:ncf}) allows one to sum over the  ``critical values'' $p_1, p_2, \dots, p_k$ instead.
To make this rigorous, we will construct a special approximation $F_n \to F$ by finite Blaschke products with critical values sets $\{p_1, p_2, \dots, p_k\}$.

\allowdisplaybreaks

Assuming the existence of such an approximating sequence, the argument runs as follows:
since the entropy can only decrease after taking limits \cite[Theorem 4.2]{inner},
\begin{align*}
\frac{1}{2\pi} \int_{|z|=1} \log|F'(z)| d\theta & \le \liminf_{n \to \infty} \frac{1}{2\pi} \int_{|z|=1} \log|F_n'(z)| d\theta, \\
& \le \liminf_{n \to \infty} \biggl \{ \log |F_n'(0)| + \sum_{i=1}^k \sum_{F_n(q_i) = p_i} \log \frac{1}{|q_i|} \biggr \}, \\
& =  \log|F'(0)| + \sum_{i=1}^k \log \frac{1}{|p_i|}.
\end{align*}
However, by (\ref{eq:sing-comparison}), the other direction is automatic:
\begin{align*}
\frac{1}{2\pi} \int_{|z|=1} \log|F'(z)| d\theta & = \lim_{r \to 1} \frac{1}{2\pi} \int_{|z|=r} \log|F'(z)| d\theta + \sigma(F'), \\
& \ge \log|F'(0)| +  \sum_{i=1}^k \sigma(F_{p_i}), \\
& = \log|F'(0)|  +   \sum_{i=1}^k \log \frac{1}{|p_i|}.
\end{align*}
Logic dictates that the sequence $F_n \to F$ is stable and the formula (\ref{eq:pomm}) holds.

\subsection{Construction of the approximating sequence}

For the construction of the approximating sequence, we employ the gluing technique of  Stephenson \cite{stephenson2},  also see the paper of Bishop \cite{bishop}.
For each puncture $p_i$, choose a real-analytic arc which joins $p_i$ to a point on the unit circle, so that the arcs are disjoint and do not pass through the origin.
Define a {\em tile} or {\em sheet} to be the shape $\mathbb{D} \setminus \cup_{i=1}^k \gamma_i$.
Let $\Gamma = \langle  g_1, g_2, \dots, g_k \rangle$ be the free group on $k$ generators. Consider the countable collection $\{T_g\}_{g \in \Gamma}$ of tiles indexed by elements of $\Gamma$. We form a simply-connected Riemann surface $S$ by gluing the lower side of $\gamma_i$ in $T_g$ to the upper side of $\gamma_i$ in $T_{g_i g}$. The surface $S$ comes equipped with a natural projection to the disk $\mathbb{D}$ which sends a point in a tile $T_g$ to its representative in the model $\mathbb{D} \setminus \cup_{i=1}^k \gamma_i$. We may uniformize $S \cong \mathbb{D}$ by taking $0$ in the base tile $T_e$ to $0$. In this uniformizing coordinate, the projection $F$ becomes a holomorphic self-map  of the disk. Since all the slits have been glued up, $F$ is an inner function, and a little thought shows that it is the universal covering map of $\mathbb{D} \setminus \{p_1,p_2,\dots,p_k\}$.

We now give a slightly different description of the above construction. For this purpose, we need the notion of an
{\em $\infty$-stack}\,: a countable collection of tiles $\{T_j\}_{j \in \mathbb{Z}}$, where the lower side of $\gamma_i$ in $T_j$ is identified with the upper side of $\gamma_{i}$ in $T_{j+1}$. To highlight the dependence on the curve $\gamma_i$, we say that the $\infty$-stack is glued over $\gamma_i$.
Similarly, by an {\em $n$-stack}\,, we mean a set of $n$ tiles with the above identifications made modulo $n$.
Now, to construct $S$, we begin with the base tile $T_e \cong \mathbb{D} \setminus \cup_{i=1}^k \gamma_i$, and at each slit $\gamma_i \subset T_e$, we glue an $\infty$-stack (i.e.~we add the tiles $\{T_j\}_{j \in \mathbb{Z} \setminus \{0\}}$ and treat $T_e$ as $T_0$).
 We refer to the tiles that were just added as the tiles of generation 1. To each of the $k-1$ unglued slits in each tile of generation 1, we glue a further $\infty$-stack of tiles, which we call tiles of generation 2. Repeating this construction infinitely many times gives the  Riemann surface $S$ from before. 

For the finite approximations, we slightly modify the above procedure. We begin with a base tile $T_e \cong \mathbb{D} \setminus \cup_{i=1}^k \gamma_i$ with $k$ slits. At each of these $k$ slits, we glue in an $n$-stack of sheets (sheets of generation 1). At each of the $k-1$ unresolved slits of sheet of generation 1, we  glue in a further $n$-stack (sheets of generation 2). We repeat for $n$ generations. Finally, at sheets of generation $n$, we resolve the slits by simply sowing their edges together. This gives us a Riemann surface $S_n$ and a finite Blaschke product $F_n$ with critical values
$p_1,p_2,\dots,p_k$.

Since the Riemann surfaces $S_n \to S$ converge in the Carath\'eodory topology, the maps $F_n \to F$ converge uniformly on compact sets. With the construction of the special approximating sequence, the proof of Theorem \ref{pomm} is complete (when the number of punctures is finite).

\subsection{Proof of Theorem A.2 when $P$ is infinite}

We handle the infinite case by reducing it to the finite case. This is achieved by the following lemma:

\begin{lemma}
Suppose that  $\mathcal U_{ P}$ is an inner function. Then,
\begin{equation}
\label{eq:mono}
\frac{1}{2\pi} \int_{|z|=1} \log|\mathcal U'_P(z)| d\theta
\ge
\frac{1}{2\pi} \int_{|z|=1} \log|\mathcal U'_Q(z)| d\theta,
\end{equation}
for any  $Q \subseteq P$.
\end{lemma}

\begin{proof}
 Topological considerations allow us to factor $\mathcal U_P = \mathcal U_Q \circ h$, where
$h$ is a holomorphic map of the disk. The normalizations $\mathcal U_P(0) = \mathcal U_Q(0) = 0$
imply that  $h(0) = 0$. Since $\mathcal U_P$ is inner, $h$ must also be inner.
The chain rule and the $h$-invariance of Lebesgue measure give
$$
\frac{1}{2\pi} \int_{|z|=1} \log|\mathcal U'_P(z)| d\theta
=
\frac{1}{2\pi} \int_{|z|=1} \log|\mathcal U'_Q(z)| d\theta
+
\frac{1}{2\pi} \int_{|z|=1} \log| h'(z)| d\theta
$$
Since $h$ is inner and $h(0) = 0$, $|h'(z)| \ge 1$ for $z \in \mathbb{S}^1$, see e.g.~\cite[Theorem 4.15]{mashreghi}.
Dropping second term gives (\ref{eq:mono}).
\end{proof}

\begin{proof}[Proof of Theorem \ref{pomm} when $P$ is infinite]
The above lemma shows that if $P$ is not a Blaschke sequence, then $\mathcal U_P$ cannot be an inner function of finite entropy. Conversely, if $P = \{p_1, p_2, \dots\}$ is a Blaschke sequence, then the integrals
$$
\frac{1}{2\pi} \int_{|z|=1} \log|\mathcal U'_{P_k}(z)| d\theta, \qquad P_k = \{p_1, p_2, \dots, p_k\},
$$
are increasing in $k$ and
\begin{equation}
\label{eq:mono1}
\frac{1}{2\pi} \int_{|z|=1} \log|\mathcal U'_{P}(z)| d\theta \ge
\lim_{k \to \infty} \frac{1}{2\pi} \int_{|z|=1} \log|\mathcal U'_{P_k}(z)| d\theta.
\end{equation}
Since the entropy can only decrease in the limit \cite[Theorem 4.2]{inner},  we must have equality in (\ref{eq:mono1}). This completes the proof.
 \end{proof}

\section[Appendix B. Existence of Perron hulls]{Existence of Perron hulls}
\label{sec:perron}

We now prove the existence statement in Theorem \ref{l1theorem}. The proof is a standard application of Schauder's fixed point theorem. Our exposition is inspired by \cite[Appendix]{conf-metrics}.

Recall that $G(z, \zeta) =  \log  \bigl | \frac{1-z\overline{\zeta}}{z-\zeta} \bigr |$ denotes the Green's function of the unit disk. Below, we will make use of two properties of the Green's function:
\begin{enumerate}
\item If $\mu$ is a finite measure on the unit disk, then
$$
G_\mu(z) = \frac{1}{2\pi}  \int_{\mathbb{D}} G(z, \zeta) d\mu
$$
solves the linear Dirichlet problem
\begin{equation}
\label{eq:LDP}
   \left\{\begin{array}{lr}
        \Delta u =  \mu, &  \text{in } \mathbb{D}, \\
       u = 0, &  \text{on } \mathbb{S}^1,
        \end{array}\right.
\end{equation}
where the boundary condition is understood in terms of weak limits.

\item The function
$$
z \ \to \ \int_{\mathbb{D}} G(z, \zeta) |d\zeta|^2
$$
is uniformly bounded on $\mathbb{D}$ and tends to 0 as $|z| \to 1$.
\end{enumerate}

Property 1 follows from the fact that $\int_{|z| = r} G(z, \zeta) d\theta \to 0$ as $r \to 1^-$ uniformly in $\zeta \in \mathbb{D}$.
An easy way to check Property 2 is to use the interpretation of the  Green's function as the occupation density of Brownian motion.

\begin{proof}[Proof of Theorem \ref{l1theorem}: existence]
Let $P_h$ denote the harmonic extension of $h$ to the unit disk. Consider the closed convex set
$$\mathscr K_h \, = \, \Bigl \{v \in L^1(\mathbb{D}, |dz|^2), \  v \le P_h \Bigr \} \, \subset \,  L^1(\mathbb{D}, |dz|^2)$$
and the operator
\begin{equation}
\label{eq:shauder-operator}
(Tv)(z) = P_h(z) - \frac{1}{2\pi} \int_{\mathbb{D}} \bigl (4e^{2v(\zeta)} |d\zeta|^2 + d\mu_\zeta \bigr ) G(z, \zeta).
\end{equation}
If $v \in \mathscr K_h$, then $\nu_\zeta = \bigl (4e^{2v(\zeta)} |d\zeta|^2 + d\mu_\zeta \bigr )$ is a finite measure. By Property 2,
$G_\nu(z)$
is a positive  function  in $L^1(\mathbb{D}, |dz|^2)$, and therefore,
$T$ maps $\mathscr K_h$ into itself.
By Property 1, $G_\nu$ has zero boundary values which means that every function in the image of $T$ has boundary data $h$. In other words, $Tv$ is the unique solution of the linear Dirichlet problem
\begin{equation}
\label{eq:generalized-GCE2c}
   \left\{\begin{array}{lr}
        \Delta u = 4 e^{2v} + \mu, &  \text{in } \mathbb{D}, \\
       u = h, &  \text{on } \mathbb{S}^1.
        \end{array}\right.
\end{equation}
In particular, $u \in \mathscr K_h$ is a fixed point of $T$ if and only if $u$ solves the Gauss curvature equation with data $(\mu, h)$.

To see that the image $T(\mathscr K_h)$ is compact,  note that by Property 2, the functions
$$
\int_{\mathbb{D}} e^{2v(\zeta)} G(z, \zeta) |d\zeta|^2, \qquad v \in \mathscr K_h,
$$
are uniformly continuous on the closed unit disk $\overline{\mathbb{D}}$.
 Since we have verified the assumptions of Schauder's fixed point theorem, $T$ has a fixed point $u \in \mathscr K_h$. The proof is complete.
\end{proof}

The reader who wishes to learn more about non-linear elliptic PDEs involving measures can consult \cite{ponce-book,marcus-veron}.

\section[Appendix C. Carleson's theorem on outer functions]{Carleson's theorem on outer functions}
\label{sec:carleson-outer}

We now briefly outline the construction of an outer function $\Phi_E \in C^\infty(\overline{\mathbb{D}})$ which vanishes on a Beurling-Carleson set $E$ to infinite order. In the literature, this fact is known as Carleson's theorem, even though the original construction due to Carleson \cite{carleson}
 only gave  $\Phi_E \in C^N(\overline{\mathbb{D}})$, where $N \ge 1$ could be any positive integer.
Here, we follow the exposition from \cite[Proposition 7.11]{HKZ}.
 Recall that if $K$ is a closed subset of the circle, then $\mathcal I(K)$ denotes the collection of open arcs that make up $\mathbb{S}^1 \setminus K$.

 The construction begins by subdividing each interval $I_n \in \mathcal I(E)$ into countably many pieces $\{J_{n,k}\}_{k \in \mathbb{Z}}$ such that $J_{n,0}$ is just the middle third interval in $I_n$ and $$|J_{n,k}| = \dist(E,J_{n,k}) = \frac{1}{3 \cdot 2^{|k|}} \cdot |I_n|.$$
Inspection shows that $F = \mathbb{S}^1 \setminus \bigcup J_{n,k}$ is a Beurling-Carleson set
and that the map $\Delta: \BC \to \BC$ which sends $E$ to $F$ is continuous, that is, if $E_n \to E$ and $\|E_n\|_{\BC} \to \|E\|_{\BC}$ then $F_n \to F$ and $\|F_n\|_{\BC} \to \|F\|_{\BC}$.
It is not difficult to see that there exists a function
$
\lambda_F : \mathcal I(F) \to [1,\infty)
$
which satisfies
\begin{equation}
\lambda_F(J) \to \infty, \qquad |J| \to 0,
\end{equation}
and
\begin{equation}
\label{eq:lambda-sum}
\sum \lambda_F(J) \cdot |J| \log \frac{1}{|J|} < \infty.
\end{equation}
With help of $\lambda_F$, one can define
$$
\Phi_E(z) = \exp \biggl [ - \sum_{J \in \mathcal I(F)} \frac{\lambda_F(J) \cdot |J| \log\frac{1}{|J|} \cdot e^{i\theta_J}}{a_J - z} \biggr ],
$$
where $e^{i\theta_J}$ is the midpoint of $J$ and $a_J = r_J e^{i\theta_J}$ is the point in $\mathbb{C}$ from which $J$ is seen from a right angle. We refer the reader to \cite{HKZ} to see that $\Phi_E$ has the desired properties.
Here, we explain that one can choose $\lambda_F$ so that $\Phi_E$ depends continuously on $F$, and thus, on $E$ if $F = \Delta(E)$.

For an interval $J \in \mathcal I(F)$, it is tempting to take
$$
\lambda_F(J) = \max \biggl \{1, \, \log \frac{1}{h_F(J)} \biggr\}
$$
where
$$
h_F(J) = \sum_{J' \in \mathcal I(F):\, |J'| \le |J|} |J'| \log \frac{1}{|J'|}.
$$
With this definition, the sum (\ref{eq:lambda-sum}) is finite and its tails converge to zero uniformly:
\begin{equation}
\label{eq:estimate-on-tails}
\sum_{J \in \mathcal I(F):\, h_F(J) \le e^{-k}} \lambda_F(J) \cdot |J| \log\frac{1}{|J|} \, \le \,
 \sum_{j=k}^\infty (k+1) e^{-k},
 \end{equation}
however, the $h_F(J), \lambda_F(J)$ will not depend continuously with respect to the Beurling-Carleson set $F$, because they are sensitive to small changes in the lengths of the intervals and the entropy of $F$.

To rectify this, we smoothen out the definitions of $h(J)$ and $\lambda(J)$, that is, we define
$$
h_F(J) = \sum_{J' \in \mathcal I(F):\, |J'| < 2|J|} \psi \biggl (\frac{|J'|}{|J|}\biggr) \cdot |J'| \log \frac{1}{|J'|},
$$
where $\psi: (0,\infty) \to [0,1]$ is a smooth function such that $\psi(t) = 1$ for $t < 1$ and $\psi(t) = 0$ for $t > 2$;
$$
\lambda_F(J) = \phi \biggl (  \log  \frac{1}{|J|} \biggr ),
$$
where $\phi$ is an increasing smooth function which satisfies $\phi(t) = t$ for $t > 2$ and $\phi(t) = 1$ for $t < 1$.

In the Korenblum topology on Beurling-Carleson sets, a neighbourhood $U_{\varepsilon,\delta}$ of $F$ consists of
all Beurling-Carleson sets $F^*$ which satisfy: (i) All intervals of length $> \varepsilon$ in $F^*$ are within $\delta$ of their counterparts in $F$ and vice versa, (ii) $\bigl | \|F^* \|_{\BC} - \| F \|_{\BC} \bigr |< \varepsilon$.
From this description, it follows that $\lambda_F$ and $h_F$ are continuous in $F$. An analogue of the estimate
(\ref{eq:estimate-on-tails})
shows that $\sum \lambda_F(J) \cdot |J| \log \frac{1}{|J|}$ also varies continuously with $F$, and thus $\Phi_E$ is continuous in $E$.

\bibliographystyle{amsplain}

\end{document}